\documentclass[bj]{imsart}
\usepackage{amsmath,amsthm,amsfonts,amssymb}
\usepackage{enumerate,color,xcolor}
\usepackage{graphicx}
\usepackage{subfigure}
\usepackage{caption}
\usepackage{verbatim}

\numberwithin{equation}{section}
\theoremstyle{plain}

\newtheorem{theorem}{Theorem}

\newtheorem{proposition}[theorem]{Proposition}

\newtheorem{remark}[theorem]{Remark}

\begin{document}

\begin{frontmatter}
\title{Large deviations and applications for Markovian Hawkes processes with a large initial intensity}
\runtitle{Gao and Zhu}

\begin{aug}
\author{\fnms{Xuefeng} \snm{Gao}\thanksref{a}\ead[label=e1]{xfgao@se.cuhk.edu.hk}}
\and
\author{\fnms{Lingjiong} \snm{Zhu}\thanksref{b}\ead[label=e2]{zhu@math.fsu.edu}}

\address[a]{Department of Systems
    Engineering and Engineering Management, The Chinese University of Hong Kong, Shatin, N.T. Hong Kong.
\printead{e1}}

\address[b]{Department of Mathematics, Florida State University, 1017 Academic Way, Tallahassee, FL-32306, United States of America.
\printead{e2}}

\runauthor{Gao and Zhu}


\end{aug}

\begin{abstract}
Hawkes process is a class of simple point processes that is self-exciting
and has clustering effect. The intensity of this point process depends on its
entire past history. It has wide applications in finance, insurance, neuroscience, social networks, criminology, seismology, and many
other fields. In this paper, we study linear Hawkes process with an exponential kernel in the asymptotic regime where the initial intensity of the Hawkes process is large.
We establish large deviations for Hawkes processes in this regime as well as the
regime when both the initial intensity and the time are large.
We illustrate the strength of our results
by discussing the applications to insurance and queueing systems.
\end{abstract}


\end{frontmatter}

\section{Introduction}
Let $N$ be a simple point process on $\mathbb{R}$ and let $\mathcal{F}^{-\infty}_{t}:=\sigma(N(C),C\in\mathcal{B}(\mathbb{R}), C\subset(-\infty,t])$ be
an increasing family of $\sigma$-algebras. Any nonnegative $\mathcal{F}^{-\infty}_{t}$-progressively measurable process $\lambda_{t}$ with
\begin{equation*}
\mathbb{E}\left[N(a,b]|\mathcal{F}^{-\infty}_{a}\right]=\mathbb{E}\left[\int_{a}^{b}\lambda_{s}ds\big|\mathcal{F}^{-\infty}_{a}\right], \quad \text{almost surely},
\end{equation*}
for all intervals $(a,b]$ is called an $\mathcal{F}^{-\infty}_{t}$-intensity of $N$. We use the notation $N_{t}:=N(0,t]$ to denote the number of
points in the interval $(0,t]$.

A Hawkes process is a simple point process $N$ admitting an $\mathcal{F}^{-\infty}_{t}$-intensity
\begin{equation}
\lambda_{t}:=\lambda\left(\int_{-\infty}^{t-}\phi(t-s)dN_{s}\right),\label{dynamics}
\end{equation}
where $\lambda(\cdot):\mathbb{R}^{+}\rightarrow\mathbb{R}^{+}$ is locally integrable, left continuous,
$\phi(\cdot):\mathbb{R}^{+}\rightarrow\mathbb{R}^{+}$ and
we always assume that $\Vert \phi\Vert_{L^{1}}=\int_{0}^{\infty}\phi(t)dt<\infty$.
In \eqref{dynamics}, $\int_{-\infty}^{t-}\phi(t-s)dN_{s}$ stands for $\sum_{\tau<t}\phi(t-\tau)$, where
$\tau$ are the occurrences of the points before time $t$. In the literature, $\phi(\cdot)$ and $\lambda(\cdot)$ are usually referred to
as exciting function (or sometimes kernel function) and rate function respectively.
A Hawkes process is linear if $\lambda(\cdot)$ is linear and it is nonlinear otherwise.

The linear Hawkes process was first introduced by A.G. Hawkes in 1971 \cite{Hawkes, Hawkes71II}. It naturally generalizes the Poisson process and it
captures both the self--exciting\footnote{Self--exciting refers the phenomenon that the occurrence of one event increases the probability of the occurrence of further
events.} property and the clustering effect. In addition, Hawkes process is a very versatile model which is amenable to statistical analysis. These explain why it has wide applications in insurance, finance, social networks, neuroscience,
criminology and many other fields.
For a list of references, we refer to \cite{ZhuThesis}.


Throughout this paper, we assume an exponential exciting function $\phi(t):=\alpha e^{-\beta t}$ where $\alpha,\beta>0$, and a linear rate function
$\lambda(z):=\mu+z$ where the base intensity $\mu\geq 0$. That is, we restrict ourselves
to the linear Markovian Hawkes process. To see the Markov property, we define
\[Z_{t}:=\int_{-\infty}^{t}\alpha e^{-\beta(t-s)}dN_{s} = Z_0 \cdot e^{-\beta t} + \int_{0}^{t}\alpha e^{-\beta(t-s)}dN_{s} .\]
Then, the process $Z$ is Markovian and satisfies the dynamics:
\begin{equation*}\label{eq:markov-Z}
dZ_{t}=-\beta Z_{t}dt+\alpha dN_{t},
\end{equation*}
where $N$ is a Hawkes process with intensity $\lambda_t = \mu + Z_{t-}$ at time $t$. In addition, the pair $(Z, N)$ is also Markovian. For simplicity, we also assume $Z_0=Z_{0-}$, i.e., there is no jump at time zero.

In this paper we consider an asymptotic regime where $Z_{0}=n$, and $n \in \mathbb{R}^{+}$ is sent to infinity. This implies the initial intensity $\lambda_0 = \mu+Z_0$ is large for fixed $\mu$.
Our main contribution is to provide the large deviations analysis
of Markovian Hawkes processes in this asymptotic regime as well as the regime when both $Z_0$ and the time are large.
The rate functions are found explicitly. Our large deviations analysis here complement our previous results in \cite{GZ}, where we establish
functional law of large numbers and functional central limit theorems
for Markovian Hawkes processes in the same asymptotic regimes.


For simplicity, the discussions in our paper are restricted to the case when the exciting function $\phi$
is exponential, that is the Markovian case. Indeed, all the results can be extended to the case
when the exciting function $\phi$ is a sum of exponential functions.
And for the non--Markovian case, we know that any continuous
and integrable function $\phi$ can be approximated by a sum of exponential functions, see e.g. \cite{ZhuI}.
In this respect, the Markovian setting in this paper is not too restrictive.
From the application point of view, the exponential exciting function
and thus the Markovian case, together with the linear rate function,
is the most widely used due to the tractability of the theoretical analysis
as well as the simulations and calibrations. See, e.g., \cite{AJ, Ait, Dassio13, Hawkes} and the references therein.

To illustrate the strength of our results, we apply them to two examples.
 In the first example, we develop approximations for finite--horizon ruin probabilities in the insurance setting where claim arrivals are modeled by Hawkes processes. Here, the initial arrival rate of claims could be high, say, right after a catastrophe event. In the second example, we rely on our large deviations results to approximate the loss probability in a multi--server queueing system where the traffic input is given by a Hawkes process with a large initial intensity. Such a queueing system could be relevant for modeling large scale service systems (e.g., server farms with thousands of servers) with high--volume traffic which exhibits clustering.

We now explain the difference between our work and the existing literature on limit theorems of Hawkes processes, especially the large deviations.
The large--time large deviations of Hawkes processes have been extensively studied
in the literature, that is the large deviation principle for $\mathbb{P}(N_{t}/t\in\cdot)$ as $t\rightarrow\infty$.
Bordenave and Torrisi \cite{Bordenave} derived the large deviations
when $\lambda(\cdot)$ is linear and obtained a closed-form formula for the rate function.
When $\lambda(\cdot)$ is nonlinear, the lack of immigration-birth representation (\cite{Hawkes}) makes
the study of large deviations much more challenging mathematically.
In the case when $\phi(\cdot)$ is exponential, the large deviations were obtained in Zhu \cite{ZhuI}
by using the Markovian property, and $\lambda(\cdot)$ is assumed to be sublinear
so that a delicate application of minmax theorem can match the lower and upper bounds.
For the general non-Markovian case, i.e., general $\phi(\cdot)$, the large deviations was obtained at the process-level in Zhu \cite{ZhuII}. The large deviations for extensions of Hawkes processes have also
been studied in the literature, see e.g. Karabash and Zhu \cite{KarabashZhu}
for the linear marked Hawkes process, and Zhu \cite{ZhuCIR} for the Cox--Ingersoll-Ross process with Hawkes jumps
and also Zhang et al. \cite{Zhang} for affine point processes. Other than the large deviations, the central limit theorems
for linear, nonlinear and extensions of Hawkes processes have been considered in, e.g., \cite{Bacry, ZhuCLT, ZhuCIR}.
Recently, Torrisi \cite{TorrisiI,TorrisiII} studied the rate of convergence
in the Gaussian and Poisson approximations of the simple point processes with stochastic intensity,
which includes as a special case, the nonlinear Hawkes process.
The moderate deviations for linear Hawkes processes were obtained in Zhu \cite{ZhuMDP},
that fills in the gap between the central limit theorem and large deviations.
Also, the large--time limit theorems for
nearly unstable, or nearly critical Hawkes processes have been considered in Jaisson and Rosenbaum \cite{JR,JRII}.
The large--time asymptotics for other regimes are referred to Zhu \cite{ZhuThesis}.
The limit theorems considered in Bacry et al. \cite{Bacry} hold for the multidimensional Hawkes process. Indeed,
one can also consider the large dimensional asymptotics for the Hawkes process, that is, mean-field limit,
see e.g. Delattre et al. \cite{Delattre}.

We organize our paper as follows. In Section \ref{TheorySec}, we will state the main theoretical results in our paper, i.e.,
the large deviations for the linear Markovian Hawkes processes with a large initial intensity.
We will then discuss the applications of our results to two examples in Section \ref{sec:app}. We prove Theorems~\ref{thm:1} and \ref{thm:2} in Section \ref{ProofsSec}.
 Technical proofs for additional results will be presented in the online appendix due to space considerations.

\section{Main results}\label{TheorySec}

In this section we state our main results.
First, let us introduce the notation that will be used throughout the paper and introduce
the definition and the contraction principle in the large deviations theory
that will be used repeatedly in the paper.

\subsection{Notation and background of large deviations theory}

We define $\mathbb{R}^{+}=\{x\in\mathbb{R}: x>0\}$
and $\mathbb{R}_{\geq 0}=\{x\in\mathbb{R}:x\geq 0\}$.
We fix $T>0$ throughout this paper. Let us first define the following spaces:

\begin{itemize}
\item
$D[0,T]$ is defined as the space of c\`{a}dl\`{a}g functions from $[0,T]$ to $\mathbb{R}_{\geq 0}$.
\item
$\mathcal{AC}_{x}[0,T]$ is defined as the space
of absolutely continuous functions from $[0,T]$ to $\mathbb{R}_{\geq 0}$ that starts at $x$ at time $0$.
\item
$\mathcal{AC}_{x}^{+}[0,T]$ is defined as the space that
consists of all the non-decreasing functions $f\in\mathcal{AC}_{x}[0,T]$.
\end{itemize}

We also define $B_{\epsilon}(x)$ as the Euclidean ball centered at $x$ with radius $\epsilon>0$.

Before we proceed, let us give a formal definition of the large deviation principle and state the contraction principle.
We refer readers to Dembo and Zeitouni \cite{Dembo} or Varadhan \cite{VaradhanII}
for general background of large deviations and the applications.

A sequence $(P_{n})_{n\in\mathbb{N}}$ of probability measures on a topological space $X$
satisfies the large deviation principle with the speed $a_{n}$ and the rate function $I:X\rightarrow[0,\infty]$ if $I$
lower semicontinuous and for any measurable set $A$, we have
\begin{equation*}
-\inf_{x\in A^{o}}I(x)\leq\liminf_{n\rightarrow\infty}\frac{1}{a_{n}}\log P_{n}(A)
\leq\limsup_{n\rightarrow\infty}\frac{1}{a_{n}}\log P_{n}(A)\leq-\inf_{x\in\overline{A}}I(x).
\end{equation*}
Here, $A^{o}$ is the interior of $A$ and $\overline{A}$ is its closure.
The rate function $I$ is said to be good if for any $m$, the level set
$\{x:I(x)\leq m\}$ is compact.

The contraction principle concerns the behavior of large deviation principle under continuous mapping from one space to another. It states that if $(P_{n})_{n\in\mathbb{N}}$ satisfies a large deviation principle on $X$ with a good rate function $I(\cdot)$, and $F$ is a continuous
mapping from the Polish space $X$ to another Polish space $Y$, then the family $Q_n = P_n F^{-1}$
satisfies a large deviation principle on $Y$ with a good rate function $J(\cdot)$ given by
\begin{equation*}
J(y) = \inf_{x: F(x)=y} I(x).
\end{equation*}

\subsection{Large deviation analysis for large initial intensity}\label{sec:largeintensity}
In this section we state a set of results on large deviations behavior of Markovian Hawkes processes when $Z_0=n$ is sent to infinity.
Note that processes $Z$ and $N$ both depend on the initial condition $Z_{0}=n$ and
we use $Z^{n},N^{n}$ to emphasize the dependence on $Z_{0}=n$.
We consider the process $Z^n$ first.

\begin{theorem}\label{thm:1}
$\mathbb{P}\left(\left\{\frac{1}{n}Z_{t}^n,0\leq t\leq T\right\}\in\cdot\right)$
satisfies a sample-path large deviation principle on $D[0,T]$ equipped with uniform topology with the speed $n$ and the good rate function
\begin{equation}\label{eq:Ig}
I_Z(g)=\int_{0}^{T}\frac{\beta g(t)+g'(t)}{\alpha}\log\left(\frac{\beta g(t)+g'(t)}{\alpha g(t)}\right)
-\left(\frac{\beta g(t)+g'(t)}{\alpha}-g(t)\right)dt,
\end{equation}
if $g\in\mathcal{AC}_{1}[0,T]$ and $g'\geq-\beta g$, and $I_Z(g)=\infty$ otherwise.
Moreover, $\mathbb{P}(\frac{1}{n}Z_{T}^n\in\cdot)$ satisfies a scalar large deviation
principle on $\mathbb{R}^{+}$ with the good rate function
\begin{eqnarray}
J(x;T)&=&\inf_{g(T)=x} I_Z(g) \label{JxT1}\\
&=& \sup_{\theta\in\mathbb{R}}\left\{\theta x-A(T;\theta)\right\} \label{JxT2}.
\end{eqnarray}
where
$A(t;\theta)$ satisfies the ODE (Ordinary Differential Equation):
\begin{eqnarray}
A'(t;\theta)&=&-\beta A(t;\theta)+e^{\alpha A(t;\theta)}-1, \label{eq:A-eq}
\\
A(0;\theta) &=& \theta. \label{eq:A-eq-initial}
\end{eqnarray}
\end{theorem}

Four remarks are in order.
\begin{enumerate}
\item [(a)] When $g(t)=e^{(\alpha-\beta)t}$ for $t \in [0,T]$, one immediately verifies from \eqref{eq:Ig} that $I_Z(g)=0$. This is consistent
with the functional law of large numbers for $\left\{\frac{1}{n}Z_{t}^n,0\leq t\leq T\right\}$ in \cite{GZ}.
\item [(b)] Note that $g'(t)=-\beta g(t)$ for any $0\leq t\leq T$ corresponds to
$Z^n_{t}=Z^n_{0}e^{-\beta t}=ne^{-\beta t}$ for any $0\leq t\leq T$,
which is equivalent to $N^n_{T}=0$. We can compute that $\mathbb{P}(N^n_{T}=0|Z^n_{0}=n)=e^{-\int_{0}^{T}(\mu+ne^{-\beta t})dt}$, which gives $-\lim_{n\rightarrow\infty}\frac{1}{n}\log\mathbb{P}(Z^n_{t}=ne^{-\beta t},0\leq t\leq T)
=\int_{0}^{T}e^{-\beta t}dt$
which is consistent with $I_{Z}(g)=\int_{0}^{T}e^{-\beta t}dt$
for $g'(t)=-\beta g(t)$ for any $0\leq t\leq T$.
\item [(c)] We have used
$A(t;\theta)$ instead of $A(t)$ to emphasize that $A$ takes value $\theta$ at time zero, and the derivative in \eqref{eq:A-eq} is taken with respect to $t$.
\item [(d)] We have two equivalent expressions for the rate function $J$: the first expression \eqref{JxT1} is directly implied by the sample--path large deviation principle
together with the contraction principle, and the second expression \eqref{JxT2} is obtained via G\"{a}rtner--Ellis Theorem. See Section~\ref{ProofsSec} for more details. In general, there are no analytical formulas for $A$ and the rate function $J$. But one can easily numerically solve the ODE for $A$ (e.g., Runge--Kutta methods) and then solve the optimization problem in \eqref{JxT2} to obtain the rate function $J$. An illustrative example is given in Figure~\ref{fig:J}.
\end{enumerate}

\begin{figure}[h]
\centering     
\subfigure[\text{$J(x;T)$ as a function of $x$. $T=5$ is fixed.}]{\includegraphics[width=.47\linewidth, height=5.35cm]{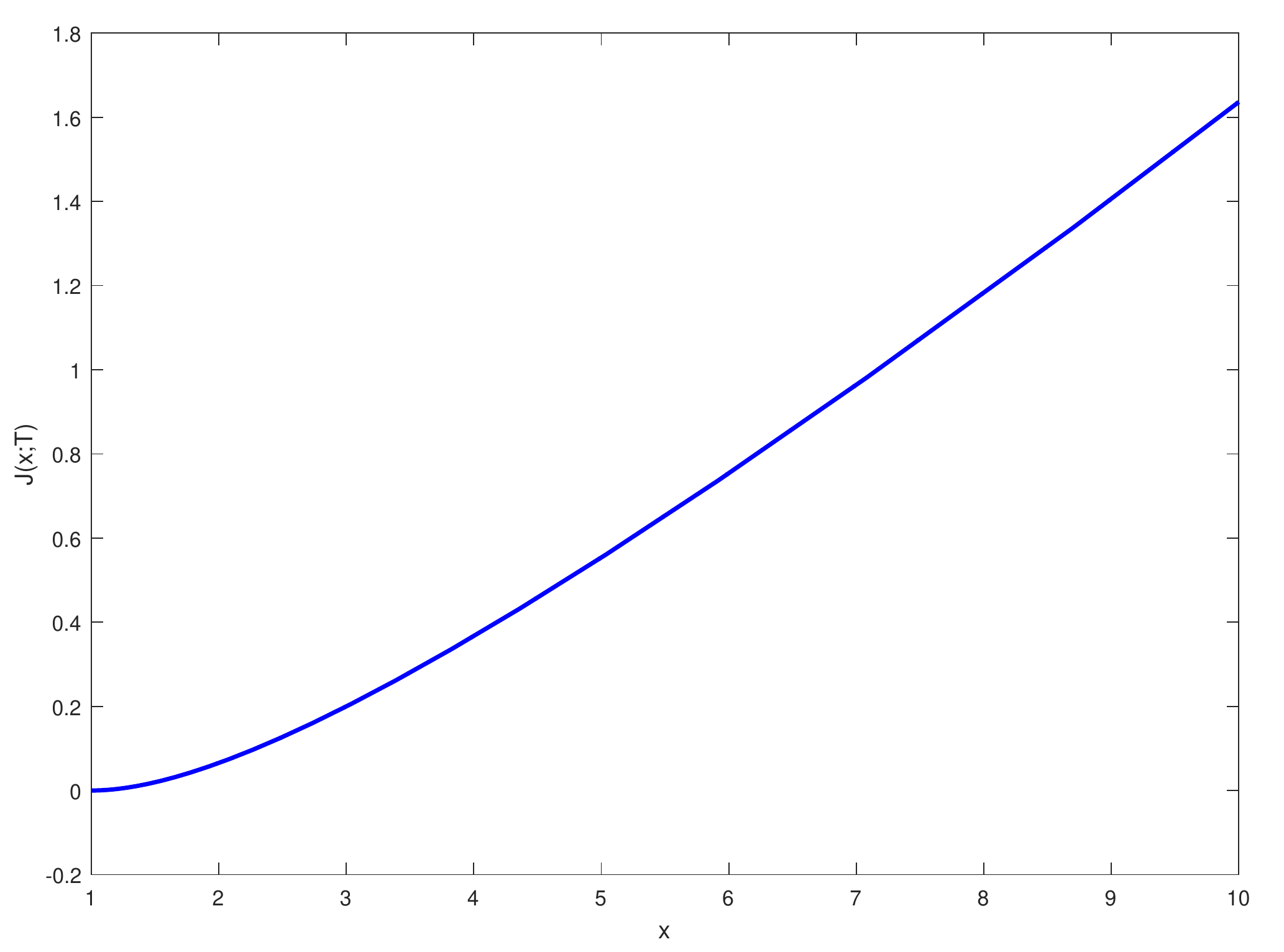}}
\subfigure[\text{$J(x;T)$ as a function of $T$. $x=3$ is fixed.}]{\includegraphics[width=.47\linewidth, height=5.5cm]{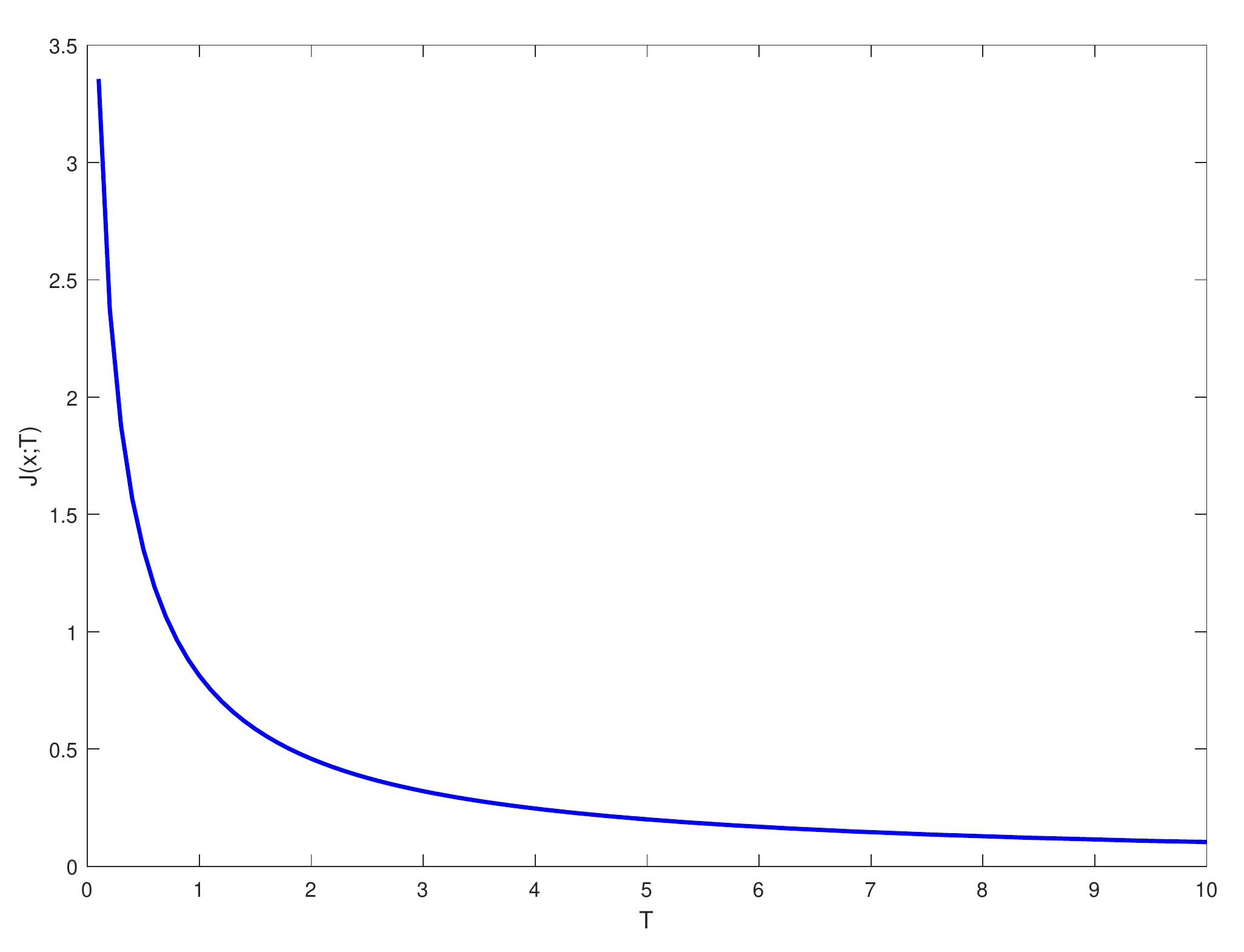}}
\caption{This figure plots the rate function $J(x;T)$ in \eqref{JxT2}. The parameters are given by: $\alpha=\beta =1$.}
\label{fig:J}
\end{figure}

Next we proceed to state a large deviation principle for $\mathbb{P}\left(\left\{\frac{1}{n}N_{t}^n,0\leq t\leq T\right\}\in\cdot\right)$.
To gain some intuition about the result, we note that
\begin{equation*}
dZ_{t}=-\beta Z_{t}dt+\alpha dN_{t},
\end{equation*}
which implies that
\begin{equation*}
N_{t}=\frac{Z_{t}-Z_{0}}{\alpha}+\frac{\beta}{\alpha}\int_{0}^{t}Z_{s}ds.
\end{equation*}
Given $Z_0=n$, equivalently we have
\begin{equation}\label{eq:Z-N}
\frac{1}{n} N_{t}^n=\frac{1}{\alpha} \cdot \left(\frac{Z_{t}^n}{n} -1 \right)+\frac{\beta}{\alpha}\int_{0}^{t} \frac{Z^n_{s}}{n}ds.
\end{equation}
Now if we define for $ t \in [0, T]$,
\[h(t)= \frac{g(t)-1}{\alpha}+\frac{\beta}{\alpha}\int_{0}^{t}g(s)ds, \]
then one readily verifies that the map $g \mapsto h$
is a continuous map from $D[0,T]$ to $D[0,T]$ under the uniform topology.
Therefore, by Theorem~\ref{thm:1} and the contraction principle, we can obtain the following result. The details of the proof is left to Section~\ref{ProofsSec}.

\begin{theorem} \label{thm:2}
$\mathbb{P}\left(\left\{\frac{1}{n}N_{t}^n,0\leq t\leq T\right\}\in\cdot\right)$
satisfies a sample-path large deviation principle on $D[0,T]$ equipped with uniform topology
with the speed $n$ and the good rate function
\begin{align}\label{INRate}
I_{N}(h)
&=\int_{0}^{T}h'(t)\log\left(\frac{h'(t)}{e^{-\beta t}+e^{-\beta t}\int_{0}^{t}\alpha e^{\beta s}h'(s)ds}\right)
\\
&\qquad\qquad\qquad
-\left(h'(t)-e^{-\beta t}-e^{-\beta t}\int_{0}^{t}\alpha e^{\beta s}h'(s)ds\right)dt,
\nonumber
\end{align}
if $h\in\mathcal{AC}_{0}^{+}[0,T]$, and $I_{N}(h)=\infty$ otherwise.
Moreover, $\mathbb{P}(N^n_T/n\in\cdot)$ satisfies
a scalar large deviation principle on $\mathbb{R}_{\geq 0}$ with the good rate function
\begin{eqnarray}
H(x;T)&=&\inf_{h:h(T)=x}I_{N}(h)\label{INxT} \\
&=& \sup_{\theta\in\mathbb{R}}\left\{\theta x-C\left(T;\frac{\theta}{\alpha}\right)+\frac{\theta}{\alpha}\right\}, \label{INxT2}
\end{eqnarray}
where $C(t;\frac{\theta}{\alpha})$ solves the ODE
\begin{eqnarray}
C'\left(t;\frac{\theta}{\alpha}\right)&=&-\beta C\left(t;\frac{\theta}{\alpha}\right) +e^{\alpha \cdot C\left(t;\frac{\theta}{\alpha}\right)}-1+\frac{\beta\theta}{\alpha}, \label{eq:C1}
\\
C\left(0;\frac{\theta}{\alpha}\right)&=&\frac{\theta}{\alpha}. \label{eq:C2}
\end{eqnarray}
\end{theorem}

Four remarks are in order.
\begin{enumerate}
\item [(a)] Notice that $I_{N}(h)=\int_{0}^{T}R(h'(t),e^{-\beta t}+e^{-\beta t}\int_{0}^{t}\alpha e^{\beta s}h'(s)ds)dt$,
where $R(x,y):=x\log\left(\frac{x}{y}\right)-x+y$, for any $x,y>0$.
It is easy to see that $R(x,y)\geq 0$ and $R(x,y)=0$ if and only if $x=y$.
Therefore $I_{N}(h)=0$ if and only if $h'(t)=e^{-\beta t}+e^{-\beta t}\int_{0}^{t}\alpha e^{\beta s}h'(s)ds$ for any $0\leq t\leq T$.
Together with $h'(0)=1$, we get $h'(t)=e^{(\alpha-\beta)t}$.
With the initial condition $h(0)=0$, we get $h(t)=\int_{0}^{t}e^{(\alpha-\beta)s}ds=t$ if $\alpha=\beta$ and
$\frac{e^{(\alpha-\beta)t}-1}{\alpha-\beta}$ if $\alpha\neq\beta$.
This is consistent
with the functional law of large numbers for $\left\{\frac{1}{n}N^n_{t},0\leq t\leq T\right\}$ in \cite{GZ}.
\item [(b)] Note that $h\equiv 0$ corresponds to $N^n_{T}=0$. We can compute that $\mathbb{P}(N^n_{T}=0|Z^n_{0}=n)=e^{-\int_{0}^{T}(\mu+ne^{-\beta t})dt}$,
which gives $-\lim_{n\rightarrow\infty}\frac{1}{n}\log\mathbb{P}(N^n_{T}=0|Z^n_{0}=n)=\int_{0}^{T}e^{-\beta t}dt$,
which is consistent with $I_{N}(h)=\int_{0}^{T}e^{-\beta t}dt$
for $h\equiv 0$.
\item [(c)] Similar as in Theorem~\ref{thm:1}, we use
$C(t;\frac{\theta}{\alpha})$ instead of $C(t)$ to emphasize that $C$ takes value $\frac{\theta}{\alpha}$ at time zero. The derivative in \eqref{eq:C1} is taken with respect to $t$.
\item [(d)] Similar as in Theorem~\ref{thm:1}, we have two equivalent expressions for the rate function $H$. In general, there is no analytical formula for $H$. But one can easily numerically solve the ODE for $C$ (e.g., Runge--Kutta methods) and then solve the optimization problem in \eqref{INxT2} to obtain the rate function $H$. An illustrative example is given in Figure~\ref{fig:H}.
\end{enumerate}

\begin{figure}[h]
\centering     
\subfigure[\text{$H(x;T)$ as a function of $x$. $T=5$ is fixed.}]{\includegraphics[width=.47\linewidth, height=5.5cm]{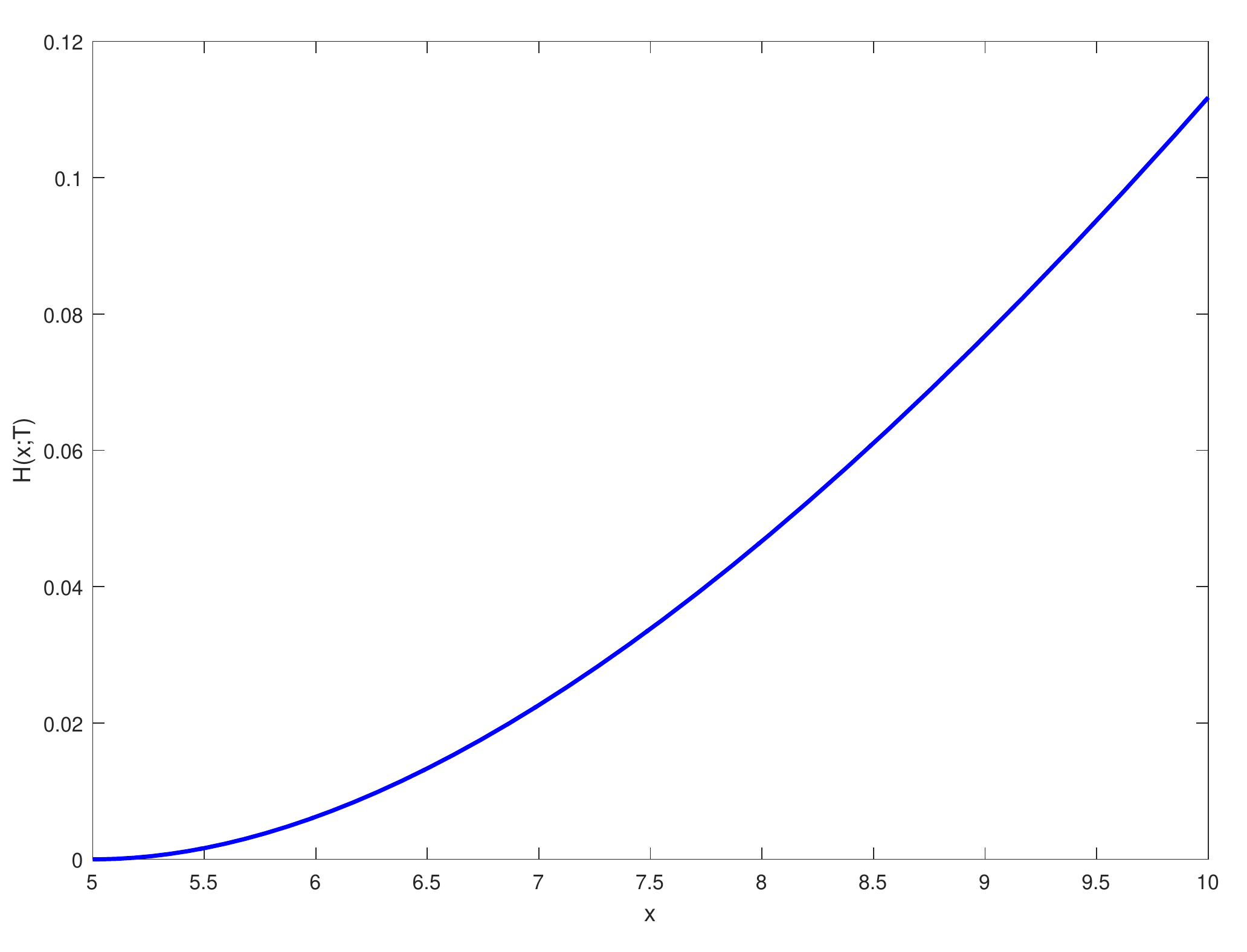}}
\subfigure[\text{$H(x;T)$ as a function of $T$. $x=5$ is fixed.}]{\includegraphics[width=.47\linewidth, height=5.4cm]{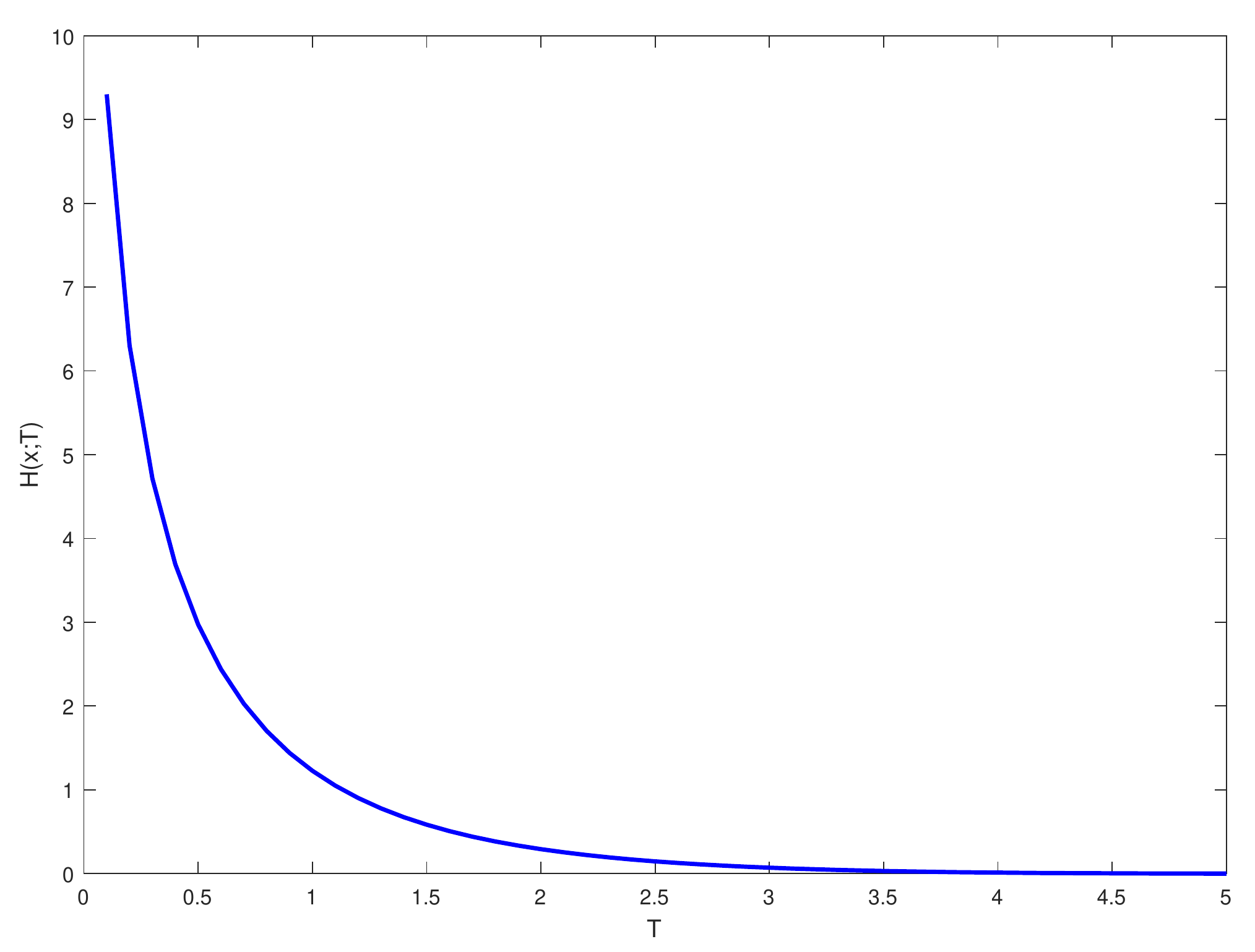}}
\caption{This figure plots the rate function $H(x;T)$ in \eqref{INxT2}. The parameters are given by: $\alpha=\beta =1$.}
\label{fig:H}
\end{figure}

\subsubsection{Most likely paths}\label{sec:path}

In this section, we compute the most likely paths to rare events for Hawkes processes with large initial intensities.
More precisely, we are interested to find the minimizer to the variational problems in \eqref{JxT1} and \eqref{INxT}.

Fix $x\in\mathbb{R}^{+}$. Let $\theta_*$ be the unique maximizer to the optimization problem \eqref{JxT2}.\footnote{It will be clear from the Proof of Theorem~\ref{thm:1} that
$A(T;\theta)=\lim_{n\rightarrow\infty}\frac{1}{n}\log\mathbb{E}[e^{\theta Z_{T}}|Z_0=n]$ if the limit exists. So one readily verifies that $A(T;\theta)$ is convex in $\theta$, and in fact strictly convex in $\theta$
from \eqref{eq:A-eq}.
Hence, there is a unique optimal $\theta_{\ast}$ for the optimization problem \eqref{JxT2}.}

\begin{proposition} \label{prop:1}
The minimizer to the variational problem \eqref{JxT1} is given by
\begin{eqnarray}\label{eq:gstar}
g_*(t) = \exp\left( \int_{0}^t \alpha e^{\alpha A(s; \theta_*)}ds - \beta t \right),
\end{eqnarray}
for $0\leq t\leq T$, where $A(s; \theta_*)$ solves the ODE \eqref{eq:A-eq} with an initial condition $A(0; \theta_*)= \theta_*$.
\end{proposition}

Next we consider the variational problem \eqref{INxT}. Let $\hat \theta_*$ be the unique maximizer to the optimization problem \eqref{INxT2}.\footnote{It will be clear from the Proof of Theorem~\ref{thm:2} that
$C(T;\frac{\theta}{\alpha})-\frac{\theta}{\alpha}=\lim_{n\rightarrow\infty}\frac{1}{n}\log\mathbb{E}[e^{\theta N_{T}}|Z_0=n]$
is always convex in $\theta$ if the limit exists. Indeed, from the ODE \eqref{eq:C1}, the limit must be strictly convex.
Hence, there is a unique optimal $\hat{\theta}_{\ast}$ for the optimization problem \eqref{INxT2}.}

\begin{proposition}\label{prop:2}
The minimizer to the variational problem \eqref{INxT} is given by
\begin{eqnarray}\label{eq:hstar}
h_*(t)=\int_{0}^{t}\exp \left(\alpha \cdot C\left(T-s; \frac{\hat \theta_*}{\alpha}\right) + {\alpha\int_{0}^{s}e^{\alpha C(T-u; \frac{\hat \theta_*}{\alpha})}du-\beta s} \right)ds,
\end{eqnarray}
for any $0\leq t\leq T$, where $C(s; \frac{\hat \theta_*}{\alpha})$ solves the ODE \eqref{eq:C1} with the initial condition $C(0; \frac{\hat \theta_*}{\alpha}) = \frac{\hat \theta_*}{\alpha}$.
\end{proposition}

The proofs of these two propositions are deferred to the online appendix.

\subsection{Large deviation analysis for large initial intensity and large time}\label{sec:largetime}

This section is devoted to a set of results on large deviations behavior of Markovian Hawkes processes in the asymptotic regime
where both $Z_0 = n$ and the time go to infinity. The proofs of these results are deferred to the online appendix due to space considerations.

When the time is sent to infinity, Hawkes processes behave differently depending on the value of
$\Vert \phi\Vert_{L^{1}}$ (see, e.g., Zhu \cite{ZhuThesis}).
In our case,  the exciting function is exponential: $\phi(t)= \alpha e^{-\beta t}$. So we have the following three different cases: (1) critical: $\alpha= \beta$; (2) super--critical: $\alpha> \beta$; and (3) sub--critical: $\alpha < \beta$.
We study each case separately.

\subsubsection{Critical case} \label{sec:critical}

We first consider the critical case, i.e., $\alpha = \beta>0$.
\begin{theorem}\label{thm:critical}
Assume that $\alpha=\beta>0$. Let $t_{n}$ be a positive sequence that goes to infinity as $n \rightarrow \infty$ and $\lim_{n\rightarrow\infty}\frac{t_{n}}{n}=0$.

(i) For any $T>0$, $\mathbb{P}(\frac{Z^n_{t_{n}T}}{n}\in\cdot)$ satisfies
a large deviation principle on $\mathbb{R}$ with the speed $\frac{n}{t_{n}}$ and the rate function
{\begin{equation*}
\hat{I}_{Z}(x)=
\frac{2(\sqrt{x}-1)^{2}}{\alpha^{2}T}, \quad \text{if $x\ge 0,$}
\end{equation*}}
and $+\infty$ otherwise.

(ii) For any $T>0$, $\mathbb{P}(\frac{N^n_{t_{n}T}}{nt_{n}}\in\cdot)$ satisfies
a large deviation principle on $\mathbb{R}$ with the speed $\frac{n}{t_{n}}$ and rate function
\begin{equation*}
\hat{I}_{N}(x)=\sup_{\theta\in\mathbb{R}}\left\{\theta x-\Lambda(\theta)\right\},
\end{equation*}
where
\begin{equation*}
\Lambda(\theta)
=
\begin{cases}
\frac{\sqrt{-2\theta}}{{\alpha}}\tanh\left(\frac{-\alpha}{\sqrt{2}}\sqrt{-\theta}T\right) &\text{if $\theta\leq 0$,}
\\
\frac{\sqrt{2\theta}}{{\alpha}}\tan\left(\frac{\alpha}{\sqrt{2}}\sqrt{\theta}T\right)
&\text{if $\theta>0$.}
\end{cases}
\end{equation*}
\end{theorem}

The proof of this result relies on G\"{a}rtner-Ellis theorem and Gronwall's inequality for nonlinear ODEs (see, e.g., \cite[Theorem~42]{Dragomir}) which arise from the characterization of the moment generating functions of $Z_t$ and $N_t$.

\subsubsection{Super--critical case} \label{sec:supercritical}
We next state the result for the super--critical case where $\alpha>\beta>0$. {Below, we use the convention that $\infty\cdot 0=0$.} 
\begin{theorem}\label{thm:supercritical}
Assume that $\alpha>\beta>0$ and $0<T<1$. Let $t_{n}=\frac{\log n}{\alpha-\beta}$.
Then,

(i) $\mathbb{P}(\frac{Z^n_{t_{n}T}}{n^{1+T}}\in\cdot)$ satisfies
a large deviation principle on $\mathbb{R}^{+}$ with the speed $n^{T}$ and the rate function $\tilde{I}_{Z}(x)={0\cdot} 1_{x=1}+\infty\cdot 1_{x\neq 1}$.

(ii) $\mathbb{P}(\frac{N^n_{t_{n}T}}{n^{1+T}}\in\cdot)$ satisfies
a large deviation principle on $\mathbb{R}_{\geq 0}$ with the speed $n^{T}$ and the rate function $\tilde{I}_{N}(x)={0\cdot} 1_{x=\frac{1}{\alpha-\beta}}+\infty\cdot 1_{x\neq\frac{1}{\alpha-\beta}}$.
\end{theorem}

We remark that the sequence $\{t_{n}\}$ in Theorem \ref{thm:supercritical} can be taken to be more general.
We choose this particular $\{t_{n}\}$ for the simplicity of notation. Notice that
when $Z_{0}=n \rightarrow \infty$, the initial intensity is $\mu+n$ which is of the same order as $n$, and
assuming $\mu=0$, we can easily compute that $\mathbb{E}[Z^n_{t}]=ne^{(\alpha-\beta)t}$.
Thus choosing $t_{n}=\frac{\log n}{\alpha-\beta}$ gives $\mathbb{E}[Z^n_{t_{n}T}]=n^{1+T}$, which is notation-wise concise.


\subsubsection{Sub--critical case} \label{sec:subcritical}

Finally, we state the large deviations results for the sub--critical case, i.e., $\beta>\alpha>0$. 
Given $Z_{0}=z$ where $z$ is a fixed constant and under the assumption $\beta>\alpha>0$, it is well known
that as $t\rightarrow\infty$, $\frac{N_{t}}{t}\rightarrow\frac{\mu}{1-\frac{\alpha}{\beta}}$ almost surely
and $\mathbb{P}(\frac{N_{t}}{t}\in\cdot)$ satisfies a large deviation principle, see e.g. \cite{Bordenave}.
So for $Z_{0}=n$, it is natural to study the large deviations for $\frac{N^n_{nT}}{n}$. 

\begin{theorem}\label{ScalarN}
Assume that $\beta>\alpha>0$. For any $T>0$, $\mathbb{P}(\frac{N^n_{nT}}{n}\in\cdot)$ satisfies
a scalar large deviation principle on $\mathbb{R}$ with the speed $n$ and the rate function
\begin{equation}\label{IConsistent}
I(x)=x\log\left(\frac{\beta x}{\alpha x+1+\mu\beta T}\right)
-x+\frac{\alpha x+1+\mu\beta T}{\beta},
\end{equation}
for $x\geq 0$ and $I(x)=+\infty$ otherwise.
\end{theorem}

The proof of this result rely on G\"{a}rtner-Ellis theorem and asymptotic behavior of the solutions of certain nonlinear ODEs which arise from the characterization of the moment generating function of $N_t$.

\begin{remark}\label{DecompRemark}
We discuss the connections with existing results on large--time large deviations of Hawkes processes here. Since the dependence on the initial condition should be self-evident here, we omit the superscript $n$ for the processes $Z$ and $N$.
As we have discussed in \cite{GZ}, when $Z_0=n$, we can decompose $N_{t}=N_{t}^{(0)}+N_{t}^{(1)}$, where
$N^{(0)}$ is a simple point process with intensity $Z^{(0)}$, where
\begin{equation*}
dZ_{t}^{(0)}=-\beta Z_{t}^{(0)}dt+\alpha dN_{t}^{(0)},
\end{equation*}
with $Z_{0}^{(0)}=n$ and $N^{(1)}$ is a simple point process with intensity
\begin{equation*}
\lambda^{(1)}_{t}:=\mu+\int_{0}^{t}\alpha e^{-\beta(t-s)}dN_{s}^{(1)}.
\end{equation*}
That is, we can decompose the Hawkes process $N$ into the sum of $N^{(0)}$ and $N^{(1)}$,
where $N^{(0)}$ is a linear Markovian Hawkes process with zero base intensity
and initial intensity $Z_{0}^{(0)}=n$
and $N^{(1)}$ is a linear Markovian Hawkes process with nonzero base intensity $\mu>0$ and empty history, i.e., $N^{(1)}(-\infty,0]=0$.
This decomposition is valid due to the immigration-birth representation of linear Hawkes processes \cite{HawkesII}.
One of the key results from the immigration-birth representation is that
the two processes $N^{(0)}$ and $N^{(1)}$ are independent of each other.

By letting $\mu=0$ in Theorem \ref{ScalarN},
$\mathbb{P}(\frac{N_{nT}^{(0)}}{n}\in\cdot)$ satisfies a large deviation principle
with the rate function
\begin{equation*}
I^{(0)}(x)=x\log\left(\frac{\beta x}{\alpha x+1}\right)
-x+\frac{\alpha x+1}{\beta}.
\end{equation*}
On the other hand, from Bordenave and Torrisi \cite{Bordenave},
$\mathbb{P}(\frac{N_{nT}^{(1)}}{n}\in\cdot)$ satisfies a large deviation principle
with the rate function
\begin{equation*}
I^{(1)}(x)=T\left[\frac{x}{T}\log\left(\frac{\frac{x}{T}}{\mu+\frac{x}{T}\frac{\alpha}{\beta}}\right)
-\frac{x}{T}+\frac{x}{T}\frac{\alpha}{\beta}+\mu\right].
\end{equation*}
Since $N^{(0)}$ and $N^{(1)}$ are independent, we conclude that
$\mathbb{P}(\frac{N_{nT}}{n}\in\cdot)$ satisfies a large deviation principle with the rate function
\begin{equation*}
I(x)=\inf_{y+z=x}\{I^{(0)}(y)+I^{(1)}(z)\}.
\end{equation*}
Notice that $I^{(1)}(x)=\mu TI^{(0)}\left(\frac{x}{\mu T}\right) + \mu T \left(1 - \frac{1}{\beta}\right)$
and $I^{(0)}(x)$ is convex in $x$.
Hence, by Jensen's inequality, we conclude that
\begin{align*}
I(x)&=\inf_{0\leq y\leq x}\left\{I^{(0)}(x-y)+\mu TI^{(0)}\left(\frac{y}{\mu T}\right)\right\} + \mu T \left(1 - \frac{1}{\beta}\right)
\\
&=(1+\mu T)\inf_{0\leq y\leq x}
\left\{\frac{1}{1+\mu T}I^{(0)}(x-y)+\frac{\mu T}{1+\mu T}I^{(0)}\left(\frac{y}{\mu T}\right)\right\} + \mu T \left(1 - \frac{1}{\beta}\right)
\nonumber
\\
&=(1+\mu T)I^{(0)}\left(\frac{1}{1+\mu T}(x-y)+\frac{\mu T}{1+\mu T}\frac{y}{\mu T}\right) + \mu T \left(1 - \frac{1}{\beta}\right)
\nonumber
\\
&=(1+\mu T)I^{(0)}\left(\frac{x}{1+\mu T}\right)+ \mu T \left(1 - \frac{1}{\beta}\right),
\nonumber
\end{align*}
which can be easily verified to be consistent with \eqref{IConsistent}.
\end{remark}

The next result is complementary to Theorem \ref{ScalarN}.

\begin{theorem} \label{ScalarN2}
Assume that $\beta>\alpha>0$ and $\mu>0$.
Let $t_{n}$ be a positive sequence that goes to infinity as $n\rightarrow\infty$.

(i) If $\lim_{n\rightarrow\infty}\frac{t_{n}}{n}=0$, then,
for any $T>0$, $\mathbb{P}(\frac{N^n_{t_{n}T}}{n}\in\cdot)$ satisfies
a large deviation principle on $\mathbb{R}_{\geq 0}$ with the speed $n$ and the rate function
\begin{equation*}
I^{(0)}(x)=x\log\left(\frac{\beta x}{\alpha x+1}\right)
-x+\frac{\alpha x+1}{\beta}.
\end{equation*}

(ii) If $\lim_{n\rightarrow\infty}\frac{t_{n}}{n}=\infty$, then,
for any $T>0$, $\mathbb{P}(\frac{N^n_{t_{n}T}}{t_{n}}\in\cdot)$ satisfies
a large deviation principle on $\mathbb{R}_{\geq 0}$ with the speed $t_{n}$ and the rate function
\begin{equation*}
I^{(1)}(x)=T\left[\frac{x}{T}\log\left(\frac{\frac{x}{T}}{\mu+\frac{x}{T}\frac{\alpha}{\beta}}\right)
-\frac{x}{T}+\frac{x}{T}\frac{\alpha}{\beta}+\mu\right].
\end{equation*}
\end{theorem}

Let us give some intuition behind the results of Theorem \ref{ScalarN2}.
Recall the decomposition $N_{t}=N_{t}^{(0)}+N_{t}^{(1)}$ from Remark \ref{DecompRemark}.
Notice that $N_{t_{n}T}^{(1)}$ is of order $t_{n}$ and that is because of the large--time law of large numbers of
the linear Hawkes process with a fixed initial intensity $\mu$ and empty history.
Also notice that $N_{t_{n}T}^{(0)}$ is of order $n$.
Let us explain. Notice that from $Z^{(0)}_0=n$ we obtain $\mathbb{E}[N_{t_{n}T}^{(0)}]=\int_{0}^{t_{n}T}\mathbb{E}[Z^{(0)}_{s}]ds
=n\int_{0}^{t_{n}T}e^{(\alpha-\beta)s}ds$.
As $n\rightarrow\infty$, we have $t_{n}T\rightarrow\infty$. But
$\int_{0}^{\infty}e^{(\alpha-\beta)s}ds=\frac{1}{\beta-\alpha}<\infty$ for $\beta>\alpha$.
Thus, $N_{t_{n}T}^{(0)}$ is of order $n$.
Hence, when $\lim_{n\rightarrow\infty}\frac{t_{n}}{n}=0$, $N^{(0)}$ `dominates' and we have result (i),
and when $\lim_{n\rightarrow\infty}\frac{t_{n}}{n}=\infty$, $N^{(1)}$ `dominates' and we obtain (ii).

So far we have discussed the large deviations for the process $N^n$ in the sub-critical case. We next consider the large deviations for the process $Z^n$ in the regime where
$Z_0=n$ and the time are both sent to infinity. Below, we use the convention that $0\cdot\infty=0$.
\begin{theorem}\label{subZLDP}
Assume that $\beta>\alpha>0$, $0<\gamma<1$, and $t_{n}:=\frac{\log n}{\beta-\alpha}$.
For any $0<T<1-\gamma$, $\mathbb{P}(\frac{Z^n_{t_{n}T}}{n^{1-T}}\in\cdot)$ satisfies
a scalar large deviation principle on $\mathbb{R}^{+}$ with the speed $n^{1-\gamma-T}$ and the rate function
\begin{equation*}
\bar {I}_{Z}(x)=0\cdot 1_{x=1}+\infty\cdot 1_{x\neq 1}.
\end{equation*}
\end{theorem}

We remark that similar as in Theorem~\ref{thm:supercritical}, here the sequence $\{t_{n}\}$ in Theorem \ref{subZLDP} can be taken to be more general.
We choose this particular $\{t_{n}\}$ for the simplicity of notation.


\section{Examples and Applications} \label{sec:app}

This section is devoted to two examples that apply the large deviations principle that we have
developed in the previous sections. The first example is on ruin probabilities in the insurance setting, and the second example is on the finite--horizon maximum of queue lengths in an infinite--server queue. We assume Markovian Hawkes processes can adequately model the clustering behavior of events occurring in each application. While this assumption may not be completely realistic, it enables us to illustrate the potential strength of our large deviations analysis.
Throughout this section, we write $a_n = o(n)$ as $n \rightarrow \infty$ if the sequence of numbers $a_n$ satisfies $\lim_{n \rightarrow \infty } a_n /n =0$.

\subsection{Example 1: Ruin probability in insurance risk theory}
In this example, we apply our large deviations results to approximate the finite horizon ruin probability in a risk model in insurance mathematics.

Hawkes processes have been applied to insurance settings to accommodate
the clustering arrival of claims observed in practice, see, e.g. \cite{DassiosII,Jang,Stabile,ZhuRuin}. When a natural disaster such as an earthquake occurs, the claims typically will not be reported following a constant intensity as in a homogeneous Poisson process. Instead, we expect
clustering effect in the claim arrivals after a catastrophe. In addition, the arrival rate of claims is typically high right after a catastrophe event. So one might use
Hawkes processes with large initial intensities to model such claim arrival processes, and it is of interest to study the finite horizon ruin probability in a risk model where the claim arrivals are modeled by such Hawkes processes. 


To study the ruin probability, let us consider the surplus process of the insurance company:
\begin{equation*}
X^{n}_t =X^n_0+\rho t-\sum_{i=1}^{N^{n}_t}Y_{i}.
\end{equation*}
Here,
$N^{n}$ is the claim arrival process modeled as a Hawkes process with an initial intensity $\mu+n$, and an exciting function $\phi(t)=\alpha e^{-\beta t}$; the constant $\rho>0$ is the premium rate, and we assume it is independent of $n$ for simplicity;
$\{Y_{i}\}$ are the non-negative claim sizes which are independent and identically distributed, and $\{Y_{i}\}$ is
independent of $N^{n}$ and $n$. Note that we use $N^n$ to emphasize the dependence on $Z_0=n$.

We are interested in approximating the finite horizon ruin probability $\mathbb{P}(\tau^{n}\leq T)$ for fixed $T>0$ and large $n$, where
$\tau^{n}$ is the ruin time of an insurance company and it is defined as follows:
\begin{equation*}
\tau^{n}:=\inf\{t>0:X^{n}_t\leq 0\}.
\end{equation*}
We assume that the initial surplus at time $0$ is given by
$X^{n}(0)=nx$, which is large, as $n\rightarrow\infty$.
In the usual setting of the finite horizon ruin probability
problem for the classical risk model, the ruin probability
is exponentially small when the initial surplus is large, see e.g. \cite{Asmussen}.
In our example, because
$N^{n}_t$ is of the order $n$, the ruin will occur at a finite time
with probability one.

Notice that $N^{n}$ satisfies a functional law of large numbers, see \cite{GZ},
\begin{equation*}
\sup_{0\leq t\leq T}\left|\frac{N^{n}_t}{n}-\psi(t)\right|\rightarrow 0, \quad \text{almost surely as $n\rightarrow\infty$},
\end{equation*}
where
$\psi(t):=
\frac{e^{(\alpha-\beta)t}-1}{\alpha-\beta}$ for $\alpha\neq\beta$,
and $\psi(t):=t$ for $\alpha=\beta$.
Therefore, as $n\rightarrow\infty$,
\begin{equation*}
\tau^{n}\rightarrow\tau^{\infty}
:=\inf\{t>0:x-\mathbb{E}[Y_{1}]\psi(t)=0\}, \quad \text{almost surely.}
\end{equation*}
It is easy to compute that (assuming that $(\alpha-\beta)\frac{x}{\mathbb{E}[Y_{1}]}+1>0$;
otherwise $\tau^{\infty}$ will be $\infty$.)
\begin{equation*}
\tau^{\infty}=
\begin{cases}
\frac{\log\left((\alpha-\beta)\frac{x}{\mathbb{E}[Y_{1}]}+1\right)}{\alpha-\beta} &\text{for $\alpha\neq\beta$},
\\
\frac{x}{\mathbb{E}[Y_{1}]} &\text{for $\alpha=\beta$}.
\end{cases}
\end{equation*}
{For any $T>\tau^{\infty}$, $\mathbb{P}(\tau^{n} \le T)\rightarrow 1$
as $n\rightarrow\infty$.
For any $T<\tau^{\infty}$, this probability will go to zero
exponentially fast as $n\rightarrow\infty$, and falls into
the large deviations regime. In the following we develop approximations for this probability $\mathbb{P}(\tau^{n} \le T)$.}

Let us assume that $\mathbb{E}[e^{\theta Y_{1}}]<\infty$ for any $\theta<\theta^{+}$
and $\mathbb{E}[e^{\theta Y_{1}}]=\infty$ otherwise,
where $\theta^{+}>0$ and we allow it to be $+\infty$.
We define $\mathcal{V}^{++}$ as the subspace of $D[0,\infty)$,
consisting of unbounded nonnegative increasing functions starting at zero at time zero
with finite variation over finite intervals
equipped with the vague topology,
see \cite{Puhalskii} .
A Mogulskii-type theorem says that, see e.g. Lemma 3.2. \cite{Puhalskii},
$\mathbb{P}\left(\left\{\frac{1}{n}\sum_{i=1}^{\lfloor nt\rfloor}Y_{i}, 0\leq t<\infty\right\}\in\cdot\right)$
satisfies a large deviation principle on $\mathcal{V}^{++}$ with the good rate function
\begin{equation*}
\int_{0}^{\infty}\overline \Lambda(g'_{1}(t))dt+\theta^{+}g_{2}(\infty), \quad \text{if $g=g_{1}+g_{2}\in\mathcal{V}^{++}, \quad g_{1} \in \mathcal{AC}_{0}[0,\infty)$},
\end{equation*}
where 
\begin{equation}\label{eq:lambda-bar}
\overline \Lambda(x):=\sup_{\theta\in\mathbb{R}}\left\{\theta x-\log\mathbb{E}[e^{\theta Y_{1}}]\right\},
\end{equation} 
and $g=g_{1}+g_{2}$ denotes the
Lebesgue decomposition of $g$ with respect to Lebesgue measure, where $g_2$ is the singular component and $g_{2}(\infty)=\lim_{t \rightarrow \infty} g_2(t)$.
Note that if $\theta^{+}=\infty$, then $g_{2}\equiv 0$.
Since $\{Y_{i}\}$ and $N^n$ are independent, then Theorem~\ref{thm:2} implies that
\[\mathbb{P}\left(\left\{\left(\frac{1}{n}N^n_{t},\frac{1}{n}\sum_{i=1}^{\lfloor ns\rfloor}Y_{i}\right),
0\leq t\leq T, 0\leq s<\infty\right\}\in\cdot\right)\]
satisfies a large deviation principle on $D[0,T]\times\mathcal{V}^{++}$\footnote{Here $D[0,T]$ is equipped
with Skorokhod topology. In Theorem~\ref{thm:1} and Theorem~\ref{thm:2} we proved first the large deviation principles
hold in the Skorokhod topology.} with the good rate function $I_{N}(h)+\int_{0}^{\infty}\overline \Lambda(g'_{1}(t))dt+\theta^{+}g_{2}(\infty)$, where the rate function $I_N(h)$ is given in Theorem~\ref{thm:2}.
It is easy to see that $\frac{1}{n}\sum_{i=1}^{N^n_t}Y_{i}=\frac{1}{n}\sum_{i=1}^{\lfloor n\cdot\frac{1}{n}N^n_t\rfloor}Y_{i}$.
Hence, by the continuity of the first-passage-time map, and the contraction principle, for any fixed $0<T<\tau^{\infty}$, we have
\begin{align}
\mathbb{P}(\tau^{n}\leq T)
&=e^{-n \cdot \inf_{h,g: x-g(h(T))\leq 0}\left\{I_{N}(h)+\int_{0}^{\infty}\overline \Lambda(g'_{1}(t))dt+\theta^{+}g_{2}(\infty)\right\}+o(n)}
\nonumber \\
&=e^{-n \cdot \inf_{h,g: x-g(h(T))\leq 0}\left\{I_{N}(h)+\int_{0}^{h(T)}\overline \Lambda(g'_{1}(t))dt+\theta^{+}g_{2}(h(T))\right\}+o(n)}, \label{eq:p-tau}
\end{align}
as $n\rightarrow\infty$.
We can replace $\infty$ by $h(T)$ in \eqref{eq:p-tau}
since $\overline{\Lambda}(x)\geq 0$ for any $x\geq 0$
and it is zero for $x=\mathbb{E}[Y_{1}]$ and $g_{2}$ is also non-decreasing so that $g_{2}(\infty)\geq g_{2}(h(T))$,
and thus the optimal $g$ satisfies $g'_{1}(t)=\mathbb{E}[Y_{1}]$ for $t>h(T)$
so that $\overline{\Lambda}(g'_{1}(t))=0$ for $t>h(T)$ and $g_{2}(\infty)=g_{2}(h(T))$. 

The expression \eqref{eq:p-tau} is not very informative, so we next simplify it to obtain
a more manageable expression which allows efficient numerical computations.
We can first fix $g_{2}(h(T))$ and then optimize over $g_{2}(h(T))$.
By the convexity of $\overline \Lambda(\cdot)$ and using Jensen's inequality, we obtain
\begin{equation*}
\int_{0}^{h(T)} \overline \Lambda(g'_{1}(t))dt
\geq h(T)\overline \Lambda\left(\frac{1}{h(T)}\int_{0}^{h(T)}g'_{1}(t)dt\right)
\geq h(T)\overline \Lambda\left(\frac{x-g_{2}(h(T))}{h(T)}\right),
\end{equation*}
where the second inequality is due to $x-g_{1}(h(T))-g_{2}(h(T))\leq 0$ and $\overline \Lambda(x)$ is non-decreasing
in $x$ for $x>\mathbb{E}[Y_{1}]$.
On the other hand, by considering
$g^{\ast}_{1}(t)=\frac{x-g_{2}(h(T))}{h(T)}t$, we have
\begin{equation*}
\int_{0}^{h(T)}\overline \Lambda((g^{\ast}_{1})'(t))dt
=h(T)\overline \Lambda\left(\frac{x-g_{2}(h(T))}{h(T)}\right).
\end{equation*}
This implies that \eqref{eq:p-tau} can be reduced to the following:
\begin{equation*}
\mathbb{P}(\tau^{n}\leq T)
=
e^{-n \cdot \inf_{h, z\leq x}\left\{I_{N}(h)+h(T) \overline \Lambda\left(\frac{x-z}{h(T)}\right)+\theta^{+}z\right\}+o(n)},
\end{equation*}
as $n\rightarrow\infty$. Therefore, we have
\begin{equation*}
\mathbb{P}(\tau^{n}\leq T)
=
e^{-n \cdot \inf_{y>0,z\leq x}\inf_{h: h(T)=y}\left\{I_{N}(h)+y \overline \Lambda\left(\frac{x-z}{y}\right)+\theta^{+}z\right\}+o(n)}.
\end{equation*}
To further simplify the above expression, we note from Theorem~\ref{thm:2} that $\mathbb{P}(N_T^n/n\in\cdot)$ satisfies
a large deviation principle with the rate function
\begin{eqnarray*}
H(x;T)=\inf_{h:h(T)=x}I_{N}(h)
= \sup_{\theta\in\mathbb{R}}\left\{\theta x-C\left(T;\frac{\theta}{\alpha}\right)+\frac{\theta}{\alpha}\right\},
\end{eqnarray*}
where $C$ solves the nonlinear ODE given in \eqref{eq:C1} and \eqref{eq:C2}.
Hence, we conclude that
\begin{equation*}
\mathbb{P}(\tau^{n}\leq T)
=
\exp(-n \cdot I_{\tau}(x; T)+o(n)) \quad \text{as $n \rightarrow \infty,$}\label{eq:simplify}
\end{equation*}
where
\begin{eqnarray*}
I_{\tau}(x;T) := \inf_{y>0, z\leq x}\left\{H(y;T)+y \overline \Lambda\left(\frac{x-z}{y}\right)+\theta^{+}z\right\}.
\end{eqnarray*}
We remark that the function $H(y;T)+y \overline \Lambda\left(\frac{x-z}{y}\right)+\theta^{+}z$ is convex in $y$. This is because
$H(y;T)$ is convex in $y$ and one can also verify directly from the convexity of $\overline \Lambda$ that $y \overline \Lambda\left(\frac{x}{y}\right)$ in convex in $y$. It is also clear that $H(y;T)+y \overline \Lambda\left(\frac{x-z}{y}\right)+\theta^{+}z$ is convex in $z$.
So we can numerically obtain $I_{\tau}(x; T)$ efficiently.


We now present a numerical example when $Y_{i}$ has a Poisson distribution with rate $1$. Then it is easy to see from \eqref{eq:lambda-bar} that
$\bar \Lambda(v)=v\log v-v+1$ for $v>0$ and $\bar \Lambda(v)=+\infty$ otherwise.
Also in this case $\theta^{+}=\infty$.
Hence, we obtain
\begin{eqnarray}\label{eq:I-tau}
 I_{\tau}(x;T) = \inf_{y>0}\left\{H(y;T)
+y \cdot \left( \frac{x}{y} - 1 - \log \left(\frac{x}{y}\right) \right) \right\}.
\end{eqnarray}
See Figure~\ref{fig:I} for a numerical illustration.



\begin{figure}[h]
\centering     
\subfigure[\text{$I_{\tau}(x;T)$ as a function of $T$. $x=0.5$ is fixed.}]{\includegraphics[width=.47\linewidth, height=5cm]{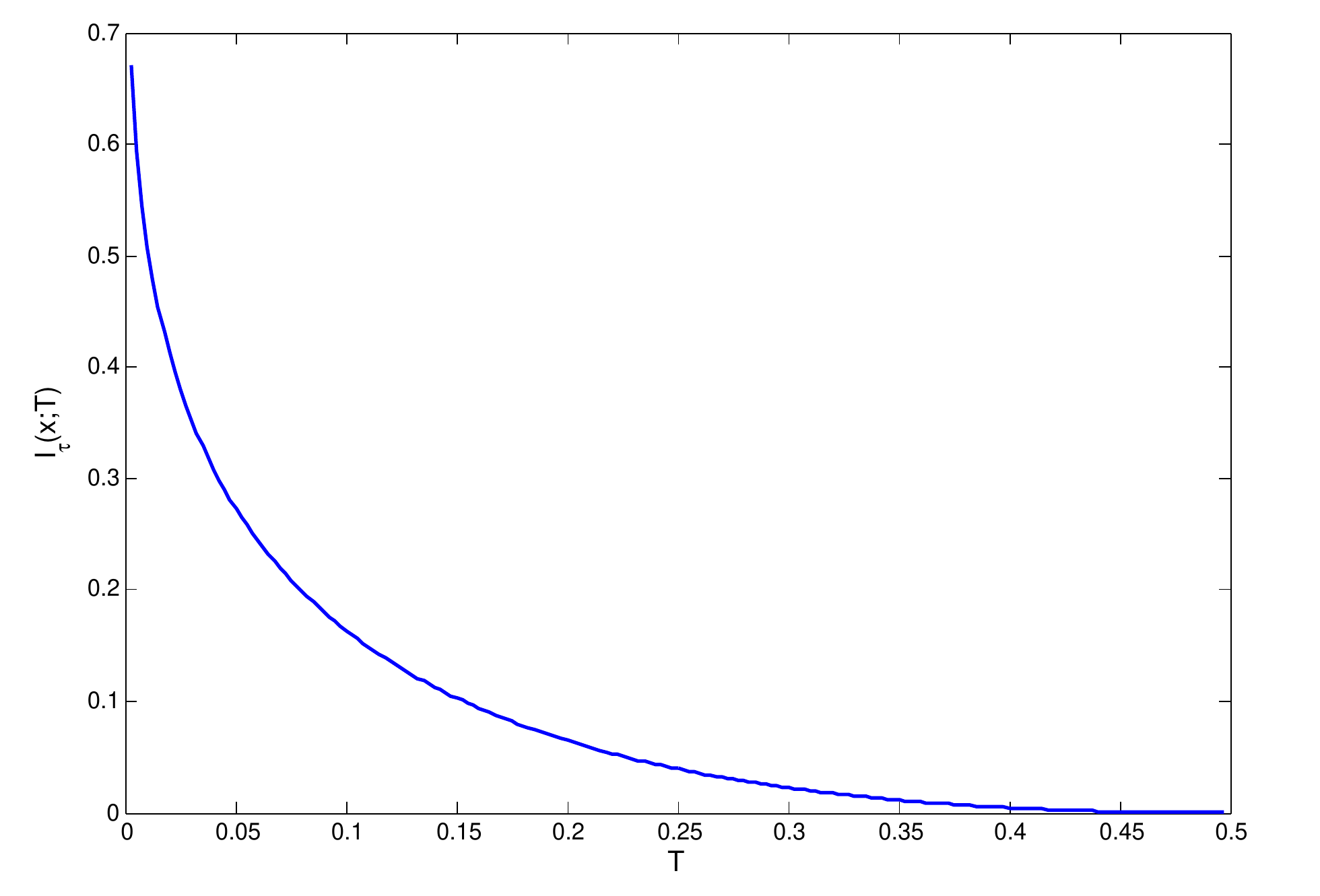}}
\subfigure[\text{$I_{\tau}(x;T)$ as a function of $x$. $T=0.2$ is fixed.}]{\includegraphics[width=.47\linewidth, height=5cm]{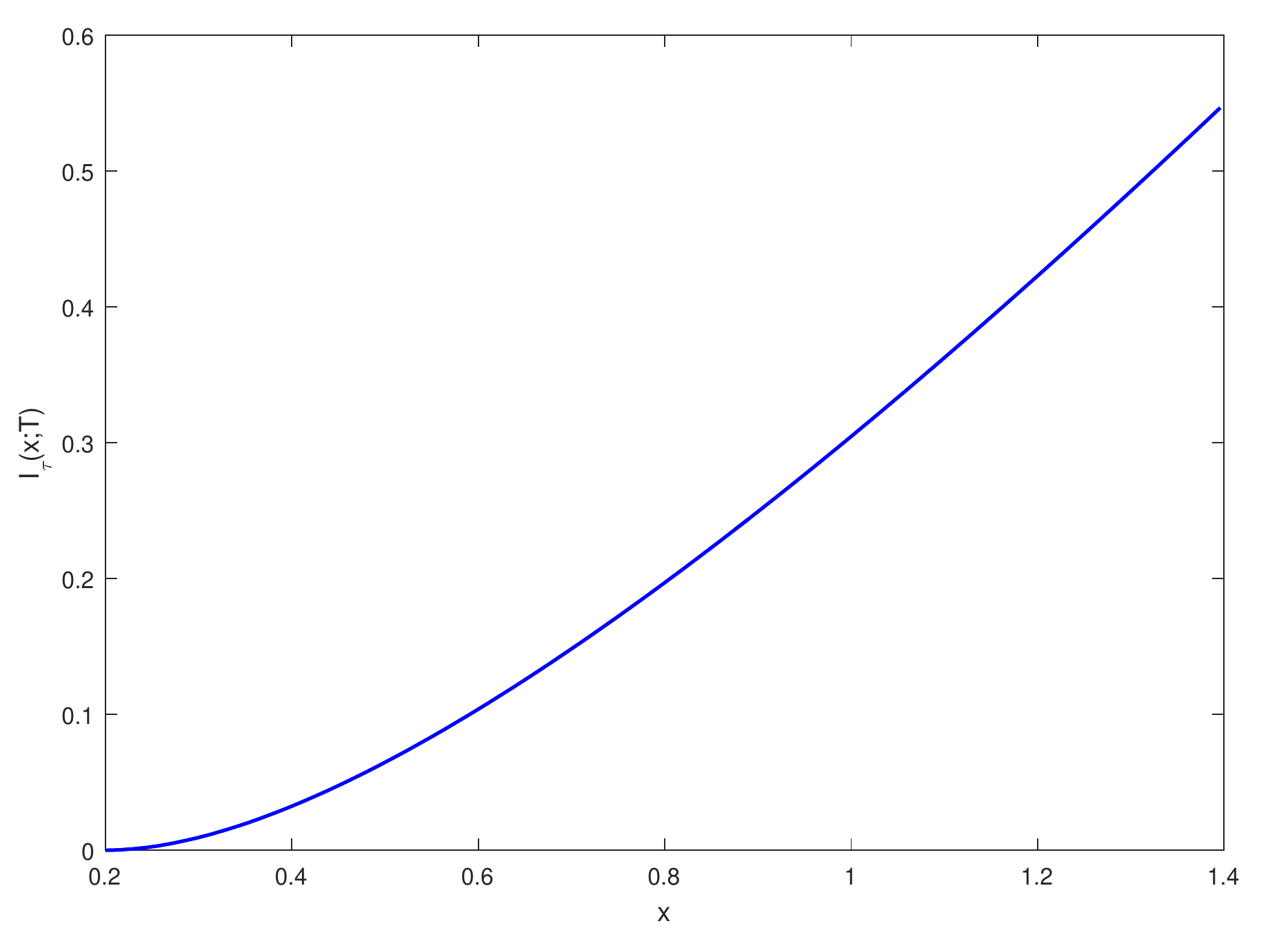}}
\caption{This figure plots $I_{\tau}(x;T)$ in \eqref{eq:I-tau}. The parameters are given by: $\alpha=\beta =1$.}
\label{fig:I}
\end{figure}




\subsection{Example 2: Finite--horizon maximum of the queue length process in an infinite--server queue}

In this example, we use our large deviations results to study certain tail probabilities in an infinite--server queue in heavy traffic where the job arrival process is modeled by a Hawkes process with a large initial intensity. Such a queueing system could be relevant for analyzing the performance of large scale service systems with high--volume traffic which exhibits clustering.
For background on infinite--server queues, their engineering applications and related large deviation analysis, see, e.g., \cite{Glynn95, Whitt02, Blanchet14}. 

Consider a sequence of queueing systems indexed by $n$ with infinite number of servers.
Jobs arrive to the $n$-th system according to a Markovian Hawkes process $N^n$ with an initial intensity $\mu + n$, and an exciting function $\phi(t)=\alpha e^{-\beta t}$. We use $N^n$ to emphasize the dependence on $Z_0=n$.
For simplicity, we assume that (a) $n$ is large so the offered load in the system is high; (b) the system is initially empty; (c) the processing time of each job is deterministic given by a constant $c>0$.

We are interested in the finite--horizon maximum of queue length process in such an infinite--serve queue, similarly as in \cite{Blanchet14}. Mathematically, we want to develop large deviations approximations for the probability of the event
\begin{equation}\label{eq:tail}
\max_{0 \le s \le T} Q^n_s \ge n x
\end{equation}
for fixed $T>0$ and sufficiently large $x$, as $n \rightarrow \infty$. Here $Q^n_s$ is number of jobs (or busy servers) in the $n$-th system at time $s$. {For sufficiently large $x$, we note that \eqref{eq:tail} is a rare event.} This event corresponds precisely to the event of observing a loss
in a queue with $nx$ servers, no waiting room, and starting empty.

It is well known that (see, e.g., \cite{Glynn91}) for the $n$-th system with deterministic processing time $c$, the queue length process $Q^n$ can be represented by
\[Q^n_t = N^n_t - N^n_{t-c},\]
where $N^n_t=0$ if $t\le 0$ by convention.
It is easy to see that the function $\Phi$ mapping $y$ to $\tilde y$ where
\[\tilde y(t): = \max_{s \le t} \{y(s) -y({s-c})\},\]
is continuous under the uniform topology. Since Theorem~\ref{thm:2} states that $\mathbb{P}(\frac{1}{n} N^n\in\cdot)$ satisfies a sample path large deviation principle with the good rate function $I_N$, we can apply the contraction principle and obtain:
\begin{align}
&\lim_{n\rightarrow\infty}\frac{1}{n} \log\mathbb{P}\left(\max_{0 \le s \le T} Q^n_s \ge n x\right) \nonumber \\
&=\lim_{n\rightarrow\infty}\frac{1}{n} \log\mathbb{P}\left(\max_{0 \le s \le T} \frac{1}{n} [N^n_s -N^n_{s-c}] \ge  x\right)
\nonumber \\
&= -\inf_h \left\{I_N(h;T): \max_{s \le T} [ h(s) -h({s-c})]  \ge x \right\} ,\label{eq:inf-opt}
\end{align}
where we use the notation $I_N(h;T)$ to emphasize the dependence of $I_N$ on $T$, as can be clearly seen in \eqref{INRate}.

Therefore, to develop large deviations approximations for $\mathbb{P}\left(\max_{0 \le s \le T} Q^n(s) \ge n x\right)$,
it remains to solve the optimization problem in \eqref{eq:inf-opt}. For $T \le c,$ since $h$ is a nondecreasing function, then the infimum in \eqref{eq:inf-opt} is simply
\[\inf_{h(T) \ge x}I_{N}(h;T)=H(x;T).\]
For $T>c$, the infimum in \eqref{eq:inf-opt} is equivalent to:
\begin{equation*}
\min\left\{\inf_{0\leq s\leq c}\inf_{h:h(s)\geq x}I_{N}(h;s),\inf_{c\leq s\leq T}\inf_{h:h(s)-h(s-c)\geq x}I_{N}(h;s)\right\}.
\end{equation*}
Now, let us solve the optimization problem:
\begin{equation*}
\inf_{h:h(t)-h(t-c)\geq x}I_{N}(h;t).
\end{equation*}
Since
\begin{align*}
&\lim_{\epsilon\rightarrow 0}\lim_{n\rightarrow\infty}\frac{1}{n}\log\mathbb{P}(N^{ny}_t/n\in B_{\epsilon}(x)|Z_{0}=ny)
=-y H(x/y;t),
\\
&\lim_{\epsilon\rightarrow 0}\lim_{n\rightarrow\infty}\frac{1}{n}\log\mathbb{P}(Z^n_{t}/n\in B_{\epsilon}(y)|Z_{0}=n)
=-J(y;t),
\end{align*}
and by the Markov property, we get
\begin{align}\label{ContractionApply}
&\lim_{\epsilon\rightarrow 0}\lim_{n\rightarrow\infty}\frac{1}{n}\log\mathbb{P}
\left([N^n_t-N^n_{t-c}]/n\in B_{\epsilon}(x),Z^n_{t-c}/n\in B_{\epsilon}(y)|Z_{0}=n\right)
\\
&=-y H(x/y;c)-J(y;t-c),
\nonumber
\end{align}
and finally for sufficiently large $x$, by \eqref{ContractionApply} and  the contraction principle, we obtain
\begin{align*}
\inf_{h:h(t)-h(t-c)\geq x}I_{N}(h;t)
&=-\lim_{\epsilon\rightarrow 0}\lim_{n\rightarrow\infty}\frac{1}{n}\log\mathbb{P}
\left([N^n_t-N^n_{t-c}]/n\in B_{\epsilon}(x)|Z_{0}=n\right)
\\
&=\inf_{y>0}\left\{y H(x/y;c)+J(y;t-c)\right\}.
\nonumber
\end{align*}
Hence we conclude that the infimum in \eqref{eq:inf-opt} is equivalent to the following expression:
\begin{equation} \label{eq:G-rate}
G(x;T):=\min\left\{\inf_{0\leq s\leq c}H (x;s),\inf_{c\leq s\leq T}\inf_{y>0}\left\{y H(x/y;c)+J(y;s-c)\right\}\right\},
\end{equation}
where $H$ and $J$ are given in Theorem~1 and 2, respectively. This implies the following approximation for $T>c$ and sufficiently large $x$:
\begin{eqnarray*}
\mathbb{P}\left(\max_{0 \le s \le T} Q^n_s \ge n x\right) = \exp(-n \cdot G(x; T) + o(n)), \quad \text{as $n \rightarrow \infty.$}
\end{eqnarray*}
Since one can solve $H$ and $J$ numerically, we can then also obtain $G$ by solving the optimization problem in \eqref{eq:G-rate} numerically.  We present an example in Figure~\ref{fig:G}.




\begin{figure}[h]
\centering     
\subfigure[\text{$G(x;T)$ as a function of $x$. $T=5$ is fixed.}]{\includegraphics[width=.47\linewidth, height=5cm]{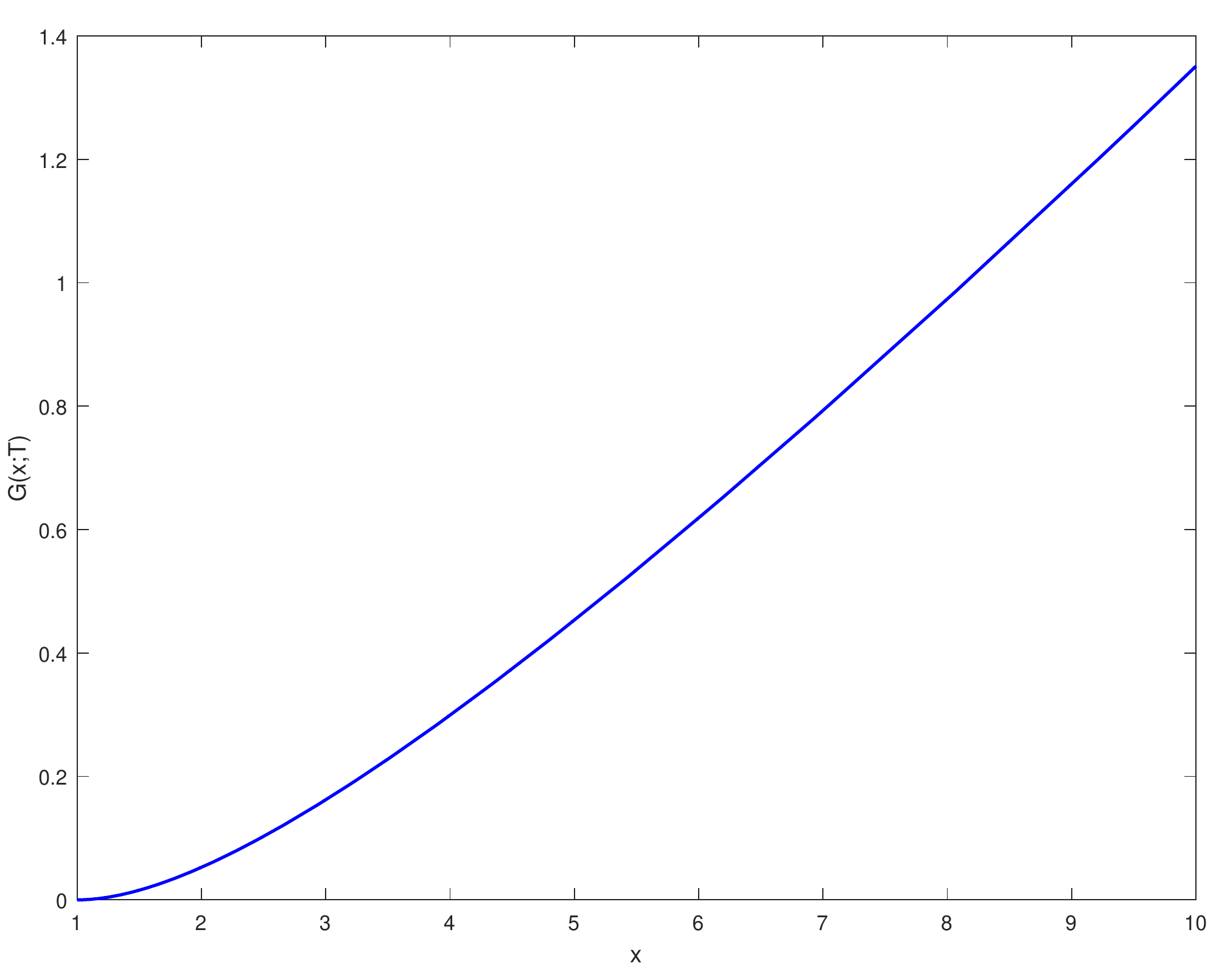}}
\subfigure[\text{$G(x;T)$ as a function of $T$. $x=5$ is fixed.}]{\includegraphics[width=.47\linewidth, height=4.9cm]{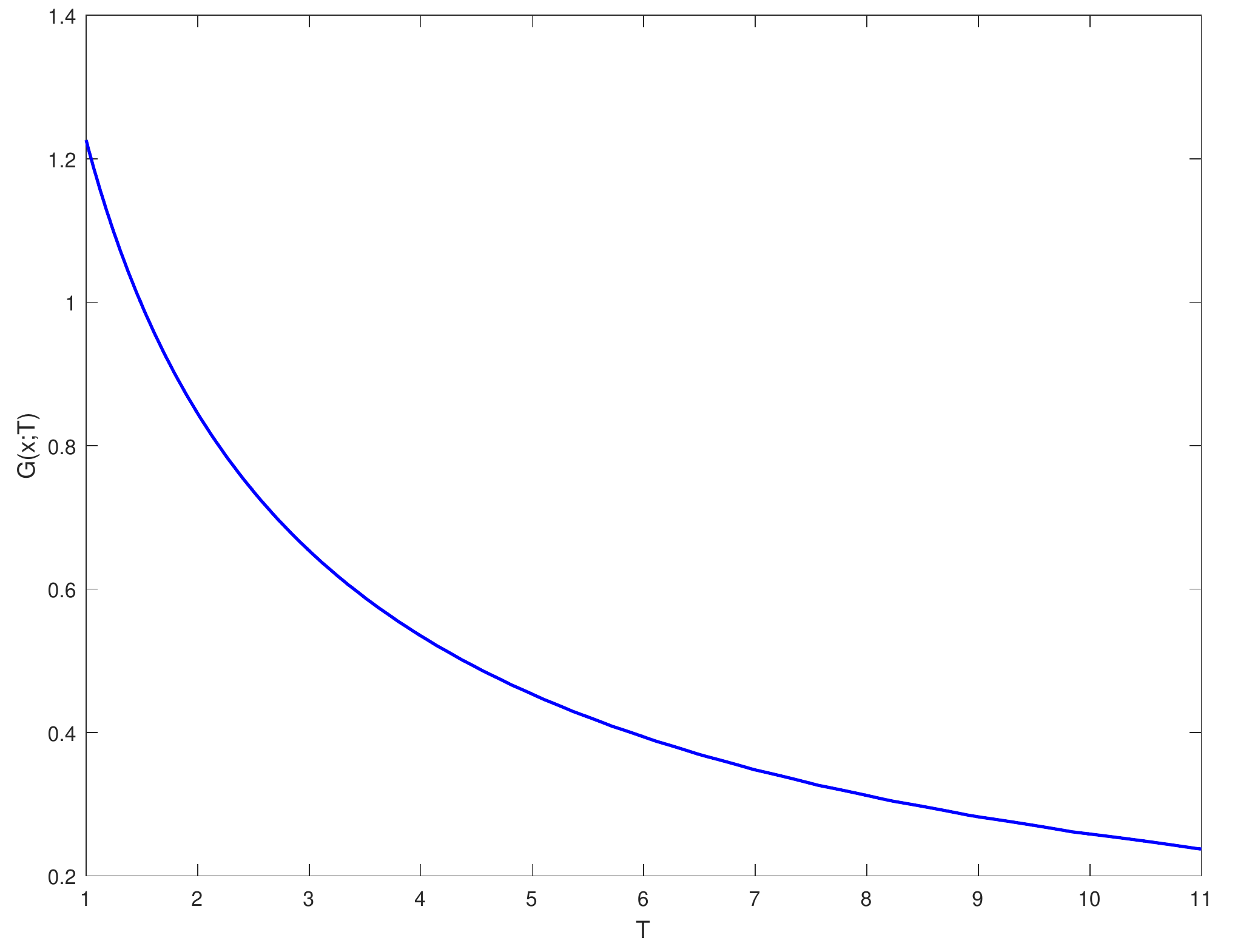}}
\caption{This figure plots $G(x;T)$ in \eqref{eq:G-rate}. We use parameters $\alpha=\beta =1, c=1$.}
\label{fig:G}
\end{figure}


\section{Proofs of Theorems~\ref{thm:1} and \ref{thm:2}}\label{ProofsSec}

This section collects the proofs of Theorems~\ref{thm:1} and \ref{thm:2}.

\subsection{Moment generating functions of $Z_t$ and $N_t$}\label{sec:MGF}
In this section we discuss the moment generating functions of $Z_t$ and $N_t$ for fixed $t$, conditioned on knowing the value of $Z_0$. These functions play a critical role in proving our large deviation results.

First, recall from \cite[Section~3.2.1]{GZ} the moment generating function of $Z_t$:
\begin{eqnarray} \label{eq:u}
u(t,z):=\mathbb{E}[e^{\theta Z_{t}}|Z_{0}=z] = e^{A(t; \theta) z + B(t; \theta)},
\end{eqnarray}
where
$A(t;\theta),B(t;\theta)$ satisfy the ODEs:
\begin{align}
&A'(t;\theta)=-\beta A(t;\theta)+e^{\alpha A(t;\theta)}-1, \label{eq:ODE-A}
\\
&B'(t;\theta)=\mu\left(e^{\alpha A(t;\theta)}-1\right), \label{eq:ODE-B}
\end{align}
with initial conditions $A(0;\theta)=\theta$ and $B(0;\theta)=0$. As remarked earlier, we have used
$A(t;\theta)$ instead of $A(t)$ to emphasize that $A$ takes value $\theta$ at time zero, and the derivative in \eqref{eq:ODE-A} is taken with respect to $t$. We also write
$B(t;\theta)$ instead of $B(t)$ to stress that $B$ depends on the initial condition of $A$.


Next, we compute the moment generating function of $N_t$. Recall that
$N_{t}=\frac{Z_{t}-Z_{0}}{\alpha}+\frac{\beta}{\alpha}\int_{0}^{t}Z_{s}ds$.
Thus, $\mathbb{E}[e^{\theta N_{t}}| Z_0 = z]=e^{-\frac{\theta}{\alpha} z} v(t,z)$, where
\[v(t,z):=\mathbb{E}[e^{\frac{\theta}{\alpha}Z_{t}+\frac{\theta\beta}{\alpha}\int_{0}^{t}Z_{s}ds} | Z_0 =z].\]
Recall that $Z$ is a Markov process with the infinitesimal generator
\begin{equation*}
\mathcal{A}f(z)=-\beta z\frac{\partial f}{\partial z}
+(z+\mu)[f(z+\alpha)-f(z)].
\end{equation*}
By Feynman-Kac formula, $v$ satisfies the equation:
\begin{equation*}
\frac{\partial v}{\partial t}
=-\beta z\frac{\partial v}{\partial z}
+(\mu+z)[v(t,z+\alpha)-v(t,z)]
+\frac{\theta\beta}{\alpha}z v(t,z),
\end{equation*}
with an initial condition $v(0,z)=e^{\frac{\theta}{\alpha}z}$.
Therefore, by the affine structure, see e.g. \cite{Errais}, one deduces that $v(t,z)=e^{C(t;\frac{\theta}{\alpha}) z+D(t;\frac{\theta}{\alpha})}$, where $C(t;\frac{\theta}{\alpha}), D(t;\frac{\theta}{\alpha})$ satisfy the ODEs:
\begin{align}
&C'\left(t;\frac{\theta}{\alpha}\right)=-\beta C\left(t;\frac{\theta}{\alpha}\right)+e^{\alpha C(t;\frac{\theta}{\alpha})}-1
+\beta \cdot C\left(0; \frac{\theta}{\alpha}\right), \label{eq:ODE-C}
\\
&D'\left(t;\frac{\theta}{\alpha}\right)=\mu\left(e^{\alpha C(t;\frac{\theta}{\alpha})}-1\right), \label{eq:ODE-D}
\end{align}
with initial conditions $C(0;\frac{\theta}{\alpha})=\frac{\theta}{\alpha}$ and $D\left(0;\frac{\theta}{\alpha}\right)=0$.
Thus we have
\begin{eqnarray} \label{eq:v}
\mathbb{E}[e^{\theta N_{t}}|Z_{0}=z] = \exp \left\{ \left(C\left(t; \frac{\theta}{\alpha}\right) -C\left(0;\frac{\theta}{\alpha}\right)\right) \cdot z
+ D\left(t; \frac{\theta}{\alpha}\right) \right\}.
\end{eqnarray}

Finally, we remark that there exists some $\Theta> 0$ such that the moment generating functions in \eqref{eq:u} and \eqref{eq:v} are both finite for all $\theta \le \Theta$. See \cite{ZhuCLT}.

\subsection{Proofs of Theorems~\ref{thm:1} and \ref{thm:2}}\label{proofof12}
We prove Theorems~\ref{thm:1} and \ref{thm:2} in this section. For notational convenience, unless specified explicitly, we use $Z$ and $N$ for $Z^{n}$ and $N^{n}$ when $Z_{0}=n$. We also use $\mathbb{E}[\cdot]$ to denote the conditional expectation $\mathbb{E}[\cdot|Z_0=n]$, and $\mathbb{P}(\cdot)$ for the conditional probability $\mathbb{P}(\cdot |Z_0=n)$.

\begin{proof}[Proof of Theorem~\ref{thm:1}]
%
%
The proof is long, so we split it into four steps.

\textbf{Step 1}.
We first establish a scalar large deviation principle for $\mathbb{P}\left(\frac{1}{n}Z_{T}\in\cdot\right)$, using G\"{a}rtner-Ellis theorem.

From \eqref{eq:u} we have
\begin{eqnarray*}
u(t,z):=\mathbb{E}[e^{\theta Z_{t}}|Z_{0}=z] = e^{A(t; \theta) z + B(t; \theta)}.
\end{eqnarray*}
It is easy to see that since $Z_{t}$ process is positive, $u(t,z)$ is monotonically increasing in $\theta$.
Let us recall from Section~\ref{sec:MGF} that
$A(t;\theta),B(t;\theta)$ satisfy the ODEs:
\begin{align*}
&A'(t;\theta)=-\beta A(t;\theta)+e^{\alpha A(t;\theta)}-1,
\\
&B'(t;\theta)=\mu\left(e^{\alpha A(t;\theta)}-1\right),
\end{align*}
with initial conditions $A(0;\theta)=\theta$ and $B(0;\theta)=0$.

Let us first consider the critical and super-critical case, that is, $\alpha\geq\beta$.
When we have $\alpha\geq\beta$, for any $A>0$, $-\beta A+e^{\alpha A}-1>0$ and
thus $A(t;\theta)$ is increasing in $t$. It is clear that for any $\theta>0$,
$\int_{\theta}^{\infty}\frac{dA}{-\beta A+e^{\alpha A}-1}<\infty$.
On the other hand, it is easy to see that $\int_{0}^{\infty}\frac{dA}{-\beta A+e^{\alpha A}-1}=\infty$.
Therefore, for any fixed $T>0$, there exists a unique positive value $\theta_{c}(T)$ such that
\begin{equation}\label{thetacT}
\int_{\theta_{c}(T)}^{\infty}\frac{dA}{-\beta A+e^{\alpha A}-1}=T.
\end{equation}
Hence, we conclude that for any fixed $T>0$, for any $0<\theta<\theta_{c}(T)$,
$A(T;\theta)$ is the unique positive value greater than $\theta$, that satisfies the equation:
\begin{equation}\label{ATtheta}
\int_{\theta}^{A(T;\theta)}\frac{dA}{-\beta A+e^{\alpha A}-1}=T.
\end{equation}
Now let us consider the case $\theta\leq 0$ .
When $\alpha>\beta$, $-\beta A+e^{\alpha A}-1=0$ when $A=0$ or
when $A=A_{c}$, for some unique negative value $A_{c}$.
For $\theta=0$ or $\theta=A_{c}$, $A(t;\theta)=0$ for any $t$.
For $A_{c}<\theta<0$, $A(t;\theta)$ is decreasing in $t$ and $A(T;\theta)$
satisfies the equation \eqref{ATtheta}.
For $\theta<A_{c}$, $A(t;\theta)$ is increasing in $t$ and $A(T;\theta)<0$
and satisfies the equation \eqref{ATtheta}.
When $\alpha=\beta$, $-\beta A+e^{\alpha A}-1>0$ when $A \ne 0$. Thus, for any $\theta<0$,
$A(t;\theta)$ is increasing in $t$ and $A(T;\theta)<0$
and satisfies the equation \eqref{ATtheta}
and also $A(t;0)\equiv 0$.
Also, it is easy to see that for $\theta<\theta_{c}(T)$, $A(t;\theta)$ is continuous and finite in $t$, and
\begin{equation*}
B(T;\theta)=\mu\int_{0}^{T}(e^{\alpha A(t;\theta)}-1)dt
\end{equation*}
is finite.
Therefore, for $\theta<\theta_{c}(T)$
\begin{equation*}
\lim_{n\rightarrow\infty}\frac{1}{n}\log
\mathbb{E}[e^{\theta Z_{T}}]=A(T;\theta).
\end{equation*}
When $\theta\geq\theta_{c}(T)$, this limit is $\infty$.
By differentiating the equation \eqref{ATtheta} with respect to $\theta$, we get
\begin{equation}\label{ATthetaDerivative}
-\frac{1}{-\beta\theta+e^{\alpha\theta}-1}+\frac{1}{-\beta A(T;\theta)+e^{\alpha A(T;\theta)}-1}\frac{d}{d\theta}A(T;\theta)=0.
\end{equation}
It is clear from the equation \eqref{thetacT} and \eqref{ATtheta} that as $\theta\rightarrow\theta_{c}(T)$,
we have $A(T;\theta)\rightarrow\infty$. Therefore, from \eqref{ATthetaDerivative}, we get
\begin{equation*}
\frac{\partial}{\partial\theta}A(T;\theta)
=\frac{-\beta A(T;\theta)+e^{\alpha A(T;\theta)}-1}{-\beta\theta+e^{\alpha\theta}-1}\rightarrow\infty,
\qquad\text{as $\theta\rightarrow\theta_{c}(T)$}.
\end{equation*}
Hence, we verified the essential smoothness condition.
By G\"{a}rtner-Ellis theorem, $\mathbb{P}\left(\frac{1}{n}Z_{T}\in\cdot\right)$
satisfies a large deviation principle on $\mathbb{R}^{+}$ with the rate function
\begin{equation}\label{eq:JxT}
J(x;T)=\sup_{\theta\in\mathbb{R}}\left\{\theta x-A(T;\theta)\right\}.
\end{equation}

Next, let us consider the sub-critical case, that is, $\alpha<\beta$. In this case,
$-\beta A+e^{\alpha A}-1=0$ if and only if $A=0$ or $A=A_{c}$, where $A_{c}$ is a positive constant
and it is unique. For $\theta=0$ or $A_{c}$, $A(t;\theta)=0$ for any $t$.
For $\theta<0$, $A(t;\theta)$ is increasing in $t$, and $A(T,\theta)<0$
satisfies the equation \eqref{ATtheta}. For $0<\theta<A_{c}$, $A(t,\theta)$ is
decreasing in $t$ and satisfies the equation \eqref{ATtheta}.
For $\theta>A_{c}$, $A(t,\theta)$ is increasing in $t$.
For any fixed $T>0$, there exists a unique $\theta_{c}(T)>A_{c}$ satisfying
the equation \eqref{thetacT} so that for any $A_{c}<\theta<\theta_{c}(T)$, $A(T,\theta)$ is the unique
positive value greater than $\theta$ that satisfies the equation \eqref{ATtheta} and for $\theta\geq\theta_{c}(T)$, $A(T,\theta)=\infty$.
We can proceed similarly as before and prove that, $\mathbb{P}\left(\frac{1}{n}Z_{T}\in\cdot\right)$
satisfies a large deviation principle on $\mathbb{R}^{+}$ with the rate function
given in \eqref{eq:JxT}.

\textbf{Step 2}.
Next, we need to prove the exponential tightness before
we proceed to establish the sample path large deviation principle.
To be more precise, we will show that
\begin{equation}\label{ClaimI}
\limsup_{K\rightarrow\infty}\limsup_{n\rightarrow\infty}
\frac{1}{n}\log\mathbb{P}\left(\sup_{0\leq t\leq T}Z_{t}\geq nK\right)=-\infty,
\end{equation}
and for any $\delta>0$,
\begin{equation}\label{ClaimII}
\limsup_{\epsilon\rightarrow 0}\limsup_{n\rightarrow\infty}
\frac{1}{n}\log\mathbb{P}\left(\sup_{|t-s|\leq\epsilon,0\leq t,s\leq T}
|Z_{t}-Z_{s}|\geq\delta n\right)=-\infty.
\end{equation}
We will also show that for any $\eta>0$,
\begin{equation}\label{ClaimIII}
\limsup_{n\rightarrow\infty}
\frac{1}{n}\log\mathbb{P}\left(\sup_{0<t\leq T}|Z_{t}-Z_{t-}|\geq\eta n\right)=-\infty.
\end{equation}
The superexponential estimates \eqref{ClaimI} and \eqref{ClaimII}
will guarantee the exponential tightness on $D[0,T]$ equipped
with the Skorokhod topology, see e.g. Theorem 4.1. in Feng and Kurtz \cite{FK}.
Together with Step 3, it will prove
the large deviation principle for $\mathbb{P}\left(\left\{\frac{1}{n}Z_{t},0\leq t\leq T\right\}\in\cdot\right)$
on $D[0,T]$ equipped with Skorokhod topology.
Next, the equation \eqref{ClaimIII}, i.e. the so-called $C$-exponentially tightness, see
e.g. Definition 4.12. in \cite{FK}
strengthens the large deviation principle for $\mathbb{P}\left(\left\{\frac{1}{n}Z_{t},0\leq t\leq T\right\}\in\cdot\right)$
so that it holds on $D[0,T]$ equipped with uniform topology, see e.g. Theorem 4.14. in \cite{FK}.

Let us first prove \eqref{ClaimI}.
Notice first that $Z_{t}-Z_{0}\leq\alpha N_{t}$ and $Z_{0}=n$.
Therefore, for $K>1$,
\begin{align*}
\mathbb{P}\left(\sup_{0\leq t\leq T}Z_{t}\geq nK\right)
&=\mathbb{P}\left(\sup_{0\leq t\leq T}(Z_{t}-Z_{0})\geq n(K-1)\right)
\\
&\leq\mathbb{P}\left(\sup_{0\leq t\leq T}N_{t}\geq n\frac{K-1}{\alpha}\right)
\nonumber
\\
&=\mathbb{P}\left(\alpha N_{T}\geq n(K-1)\right) \nonumber \\
&\leq \mathbb{E}[e^{\theta N_{T}}]e^{-\theta(K-1)n / \alpha },
\nonumber
\end{align*}
where the last inequality follows from Chebychev's inequality. In conjunction with the moment generating function of $N_T$ in \eqref{eq:v}, we hence obtain
\begin{equation*}
\limsup_{n\rightarrow\infty}\frac{1}{n}\log \mathbb{P}\left(\sup_{0\leq t\leq T}Z_{t}\geq nK\right)
\leq C\left(T;\frac{\theta}{\alpha}\right)-\frac{\theta}{\alpha} -\frac{\theta(K-1)}{\alpha},
\end{equation*}
which goes to $-\infty$ as $K\rightarrow\infty$. Hence, we proved \eqref{ClaimI}.

Next, let us prove \eqref{ClaimII}.
Note that for $s<t$, $\alpha N(s,t]=Z_{t}-Z_{s}+\beta\int_{s}^{t}Z_{u}du$.
Thus, for $s<t$, we have
\begin{equation*}
|Z_{t}-Z_{s}|
\leq\alpha N(s,t]+\beta(t-s)\sup_{s\leq u\leq t}Z_{u}.
\end{equation*}
Therefore,
\begin{align*}
&\mathbb{P}\left(\sup_{|t-s|\leq\epsilon,0\leq t,s\leq T}
|Z_{t}-Z_{s}|\geq\delta n\right)
\\
&\leq
\mathbb{P}\left(\sup_{|t-s|\leq\epsilon,0\leq s\leq t\leq T}
\left(\alpha N(s,t]+\beta(t-s)\sup_{s\leq u\leq t}Z_{u}\right)\geq\delta n\right)
\nonumber
\\
&\leq
\mathbb{P}\left(\sup_{|t-s|\leq\epsilon,0\leq s\leq t\leq T}
\alpha N(s,t]\geq\frac{\delta}{2}n\right)
+\mathbb{P}\left(\sup_{|t-s|\leq\epsilon,0\leq s\leq t\leq T}
\beta(t-s)\sup_{s\leq u\leq t}Z_{u}\geq\frac{\delta}{2}n\right).
\nonumber
\end{align*}
Note that
\begin{align*}
\mathbb{P}\left(\sup_{|t-s|\leq\epsilon,0\leq s\leq t\leq T}
\beta(t-s)\sup_{s\leq u\leq t}Z_{u}\geq\frac{\delta}{2}n\right)
\leq
\mathbb{P}\left(\beta\epsilon\sup_{0\leq u\leq T}Z_{u}\geq\frac{\delta}{2}n\right).
\nonumber
\end{align*}
By \eqref{ClaimI}, we have
\begin{equation*}
\limsup_{\epsilon\rightarrow 0}
\limsup_{n\rightarrow\infty}
\frac{1}{n}\log\mathbb{P}\left(\beta\epsilon\sup_{0\leq u\leq T}Z_{u}\geq\frac{\delta}{2}n\right)=-\infty.
\end{equation*}
Next, notice that without loss of generality we can assume that $\frac{1}{\epsilon}\in\mathbb{N}$ and
\begin{align}\label{eq:bound-P}
\mathbb{P}\left(\sup_{|t-s|\leq\epsilon,0\leq s\leq t\leq T}
\alpha N(s,t]\geq\frac{\delta}{2}n\right)
&\leq\mathbb{P}\left(\exists 1\leq j\leq T/\epsilon:
\alpha N(t_{j-1},t_{j}]\geq\frac{\delta}{4}n\right)\nonumber
\\
&\leq\sum_{j=1}^{T/\epsilon}\mathbb{P}\left(\alpha N(t_{j-1},t_{j}]\geq\frac{\delta}{4}n\right),
\end{align}
where $0=t_{0}<t_{1}<\cdots<t_{T/\epsilon}=T$, where $t_{j}-t_{j-1}=\epsilon$ for any $j$.
In addition, note that for $\theta>0$,
\begin{align} \label{eq:bound-N}
\mathbb{E}\left[e^{\theta\alpha N(t_{j-1},t_{j}]}\right]
&=\mathbb{E}\left[\mathbb{E}\left[e^{\theta\alpha N(t_{j-1},t_{j}]}|Z_{t_{j-1}}\right]\right]
\\
&=\mathbb{E}\left[e^{-\theta Z_{t_{j-1}}}e^{C(t_{j}-t_{j-1};\theta)Z_{t_{j-1}}+D(t_{j}-t_{j-1};\theta)}\right]
\nonumber
\\
&=\exp\big({D(\epsilon; \theta)+A(t_{j-1};C(\epsilon;\theta)-\theta)n+B(t_{j-1};C(\epsilon;\theta)-\theta)}\big),
\nonumber
\end{align}
where we have used the moment generating functions of $Z_t$ and $N_t$ in Section~\ref{sec:MGF}.
Hence, using Chebychev's inequality and combining \eqref{eq:bound-P} and \eqref{eq:bound-N}, we find for fixed $\epsilon>0,$
\begin{eqnarray*}
\lefteqn{\limsup_{n\rightarrow\infty}\frac{1}{n}\log
\mathbb{P}\left(\sup_{|t-s|\leq\epsilon,0\leq s\leq t\leq T}
\alpha N(s,t]\geq\frac{\delta}{2}n\right)}\\
&\leq& \sup_{1\leq j\leq T/\epsilon}
\left\{A(t_{j-1};C(\epsilon;\theta)-\theta)-\theta\frac{\delta}{4}\right\}\\
& \le & \sup_{ 0 \le t \le T}
\left\{A(t;C(\epsilon;\theta)-\theta)\right\}-\theta\frac{\delta}{4}.
\end{eqnarray*}

So in order to prove \eqref{ClaimII}, what remains is to choose $\theta$ that depends on $\epsilon$
so that (i) $\theta\rightarrow\infty$
as $\epsilon\rightarrow 0$; (ii) $A(t;C(\epsilon;\theta)-\theta)$ is uniformly bounded for $t \in [0, T]$ and $\epsilon \rightarrow 0$.
To this end, let us define $y(t) := C(t; \theta) - C(0; \theta)=C(t; \theta) - \theta$. Then $y$ satisfies the ODE:
\begin{align*}
&y'(t)=-\beta y(t)+ e^{\alpha \theta}e^{\alpha y(t)}-1,
\\
&y(0) = 0.
\end{align*}
For $\theta>0$, we have $y'(0) = e^{\alpha \theta} -1 >0,$ which implies $y$ is increasing on $[0, \gamma]$ for some $\gamma>0$. This suggests that
\[0 < y'(t) \le e^{\alpha \theta}e^{\alpha y(t)}, \quad \text{for $t \in [0, \gamma]$}.\]
By Gronwall's inequality for nonlinear ODEs, we obtain
\begin{eqnarray} \label{eq:compare-xt}
0 \le y(t) \le -\frac{1}{\alpha} \cdot \log (1 - \alpha e^{\alpha\theta} t), \quad \text{for $t \in [0, \gamma]$}.
\end{eqnarray}
Let us set $\alpha e^{\alpha\theta}  = \frac{1}{\sqrt{\epsilon}}.$ Then it is clear that $\theta\rightarrow\infty$
as $\epsilon\rightarrow 0$.
In addition, we deduce from \eqref{eq:compare-xt} that for $\epsilon< \gamma,$
\begin{eqnarray} \label{eq:x-eps}
0 \le C(\epsilon;\theta)-\theta = y(\epsilon) \le -\frac{1}{\alpha} \cdot \log (1 - \sqrt{\epsilon}).
\end{eqnarray}

Next we show $\left\{A(t;C(\epsilon;\theta)-\theta)\right\}$ is uniformly bounded for $t \in [0, T]$ and $\epsilon \rightarrow 0$. When $\alpha < \beta$, it is clear that zero is
a stable solution for the ODE of $A$ in \eqref{eq:ODE-A}. Since $A(0;C(\epsilon;\theta)-\theta))= y(\epsilon) \rightarrow 0$ as $\epsilon \rightarrow 0$, so the stability of zero solution implies that when $\epsilon \rightarrow 0$, $\left\{A(t;C(\epsilon;\theta)-\theta)\right\}$ is uniformly small and thus uniformly bounded for all $t \ge 0$. When $\alpha \ge \beta$, since $A(0 ;C(\epsilon;\theta)-\theta)) = y(\epsilon) \ge 0$, one readily checks that $A$ is non--decreasing with respect to time $t$. Hence we obtain
\[\sup_{ 0 \le t \le T}
\left\{A(t;C(\epsilon;\theta)-\theta)\right\}= A(T; y(\epsilon)). \]
We have shown in Step 1 that $A(T; \bar\theta)$ is finite when $\bar \theta<\theta_c(T)$, and $A(T; \bar \theta)$ is continuous as a function of $\bar \theta$. Therefore we deduce from \eqref{eq:x-eps} that $A(T; y(\epsilon))$ is uniformly bounded for $\epsilon \rightarrow 0.$ Thus, we have proved \eqref{ClaimII}.

Finally, the claim in \eqref{ClaimIII} trivially holds since for any $0<t\leq T$, $|Z_{t-}-Z_{t}|=0$
or $\alpha$ with probability $1$.

\textbf{Step 3}.
Next, we establish the sample path large deviation principle.

For any $\epsilon>0$, let $B_{\epsilon}(x)$ denote the open ball centered at $x$ with radius $\epsilon$.
For any $0=:t_{0}<t_{1}<t_{2}<\cdots<t_{k-1}<t_{k}:=T$ and $x_{1},\ldots,x_{k}\in\mathbb{R}^{+}$,
by the Markov property of the process $Z$, we have
\begin{align*}
&\mathbb{P}\left(\frac{1}{n}Z_{t_{1}}\in B_{\epsilon}(x_{1}),\frac{1}{n}Z_{t_{2}}\in B_{\epsilon}(x_{2}),
\ldots,\frac{1}{n}Z_{t_{k}}\in B_{\epsilon}(x_{k})\right)
\\
&=\mathbb{P}\left(\frac{1}{n}Z_{t_{1}}\in B_{\epsilon}(x_{1})\right)
\mathbb{P}\left(\frac{1}{n}Z_{t_{2}}\in B_{\epsilon}(x_{2})\bigg|\frac{1}{n}Z_{t_{1}}\in B_{\epsilon}(x_{1})\right)
\nonumber
\\
&\qquad\qquad\qquad
\cdots
\mathbb{P}\left(\frac{1}{n}Z_{t_{k}}\in B_{\epsilon}(x_{k})\bigg|\frac{1}{n}Z_{t_{k-1}}\in B_{\epsilon}(x_{k-1})\right).
\nonumber
\end{align*}
Hence, we have
\begin{align*}
&\lim_{\epsilon\rightarrow 0}
\lim_{n\rightarrow\infty}\frac{1}{n}\log\mathbb{P}\left(\frac{1}{n}Z_{t_{1}}\in B_{\epsilon}(x_{1}),
\frac{1}{n}Z_{t_{2}}\in B_{\epsilon}(x_{2}),
\ldots,\frac{1}{n}Z_{t_{k}}\in B_{\epsilon}(x_{k})\right)
\\
&=-J(x_{1};t_{1})-x_{1}J\left(\frac{x_{2}}{x_{1}};t_{2}-t_{1}\right)
-\cdots-x_{k-1}J\left(\frac{x_{k}}{x_{k-1}};t_{k}-t_{k-1}\right),
\nonumber
\end{align*}
where $J$ is given in \eqref{eq:JxT}.
Hence, for any $g\in\mathcal{AC}_{1}[0,T]$,
\begin{align*}
&\lim_{\epsilon\rightarrow 0}
\lim_{n\rightarrow\infty}\frac{1}{n}\log\mathbb{P}\left(\frac{1}{n}Z_{t_{1}}\in B_{\epsilon}(g(t_{1})),
\frac{1}{n}Z_{t_{2}}\in B_{\epsilon}(g(t_{2})),
\ldots,\frac{1}{n}Z_{t_{k}}\in B_{\epsilon}(g(t_{k}))\right)
\\
&=-J(g(t_{1});t_{1})-g(t_{1})J\left(\frac{g(t_{2})}{g(t_{1})};t_{2}-t_{1}\right)
-\cdots-g(t_{k-1})J\left(\frac{g(t_{k})}{g(t_{k-1})};t_{k}-t_{k-1}\right).
\nonumber
\end{align*}
For any given positive $g\in\mathcal{AC}_{1}[0,T]$,
we have
\begin{align*}
&J\left(\frac{g(t_{j})}{g(t_{j-1})};t_{j}-t_{j-1}\right)
\\
&=\sup_{\theta\in\mathbb{R}}\left\{\theta\frac{g(t_{j})}{g(t_{j-1})}-A(t_{j}-t_{j-1};\theta)\right\}
\nonumber
\\
&=\sup_{\theta\in\mathbb{R}}\left\{\theta\left(1+\frac{g'(t_{j-1}^{\ast})}{g(t_{j-1})}(t_{j}-t_{j-1})\right)
-\theta-\int_{0}^{t_{j}-t_{j-1}}(-\beta A(s;\theta)+e^{\alpha A(s;\theta)}-1)ds\right\}
\nonumber
\\
&=(t_{j}-t_{j-1})\sup_{\theta\in\mathbb{R}}\left\{\theta\frac{g'(t_{j-1}^{\ast})}{g(t_{j-1})}
-\left(-\beta A(t_{j-1}^{\ast\ast};\theta)+e^{\alpha A(t_{j-1}^{\ast\ast};\theta)}-1\right)\right\},
\nonumber
\end{align*}
where $t_{j-1}^{\ast}\in[t_{j-1},t_{j}]$ is independent of $\theta$
and $t_{j-1}^{\ast\ast}\in[0,t_{j}-t_{j-1}]$ may depend on $\theta$.

It is easy to see that for any given positive $g\in\mathcal{AC}_{1}[0,T]$,
$\frac{g'(t_{j}^{\ast})}{g(t_{j-1})}$, is uniformly bounded in $j$. To see this,
notice that $g$ is positive and continuous so $\inf_{0\leq t\leq T}g(t)>0$, and since $g$ is absolutely continuous, $g'$ exists
almost surely {and we can assume that $g'$ exist for any $t_{j}^{\ast}$.}
And we can also see that $A(t_{j-1}^{\ast\ast};\theta)$ is uniformly bounded in $j$.
Therefore, there exists some constant $K$ that may depend on the given $g$,
such that, uniformly in $j$,
\begin{align*}
&\sup_{\theta\in\mathbb{R}}\left\{\theta\frac{g'(t_{j-1}^{\ast})}{g(t_{j-1})}
-\left(-\beta A(t_{j-1}^{\ast\ast};\theta)+e^{\alpha A(t_{j-1}^{\ast\ast};\theta)}-1\right)\right\}
\\
&=\sup_{|\theta|\leq K}\left\{\theta\frac{g'(t_{j-1}^{\ast})}{g(t_{j-1})}
-\left(-\beta A(t_{j-1}^{\ast\ast};\theta)+e^{\alpha A(t_{j-1}^{\ast\ast};\theta)}-1\right)\right\}.
\nonumber
\end{align*}
Therefore,
\begin{align*}
&\bigg|\sup_{\theta\in\mathbb{R}}\left\{\theta\frac{g'(t_{j-1}^{\ast})}{g(t_{j-1})}
-\left(-\beta A(t_{j-1}^{\ast\ast};\theta)+e^{\alpha A(t_{j-1}^{\ast\ast};\theta)}-1\right)\right\}
\\
&\qquad\qquad
-\sup_{\theta\in\mathbb{R}}\left\{\theta\frac{g'(t_{j-1}^{\ast})}{g(t_{j-1})}
-\left(-\beta\theta+e^{\alpha\theta}-1\right)\right\}\bigg|
\nonumber
\\
&\leq
\sup_{|\theta|\leq K}\sup_{0\leq t\leq t_{j}-t_{j-1}}\left|\left(-\beta A(t;\theta)+e^{\alpha A(t;\theta)}-1\right)
-\left(-\beta\theta+e^{\alpha\theta}-1\right)\right|\rightarrow 0,
\nonumber
\end{align*}
as $t_{j}-t_{j-1}\rightarrow 0$.
Hence, we conclude that
\begin{align*}
&\lim_{\epsilon\rightarrow 0}\lim_{n\rightarrow\infty}\frac{1}{n}\log\mathbb{P}\left(\frac{1}{n}Z_{t}\in B_{\epsilon}(g), 0\leq t\leq T\right)
\\
&=-\int_{0}^{T}g(t)\sup_{\theta\in\mathbb{R}}\left\{\theta\frac{g'(t)}{g(t)}
-(-\beta\theta+e^{\alpha\theta}-1)\right\}dt
\nonumber
\\
&=-\sup_{\theta(t):0\leq t\leq T}\int_{0}^{T}\left\{\theta(t)g'(t)
-(-\beta\theta(t)+e^{\alpha\theta(t)}-1)g(t)\right\}dt.
\nonumber
\end{align*}

Together with the superexponential estimates \eqref{ClaimI} and \eqref{ClaimII},
we have proved that, $\mathbb{P}\left(\left\{\frac{1}{n}Z_{t},0\leq t\leq T\right\}\in\cdot\right)$
satisfies a large deviation principle with the rate function
\begin{equation*}
I_Z(g)=\sup_{\theta(t):0\leq t\leq T}\int_{0}^{T}\left\{\theta(t)g'(t)
-(-\beta\theta(t)+e^{\alpha\theta(t)}-1)g(t)\right\}dt,
\end{equation*}
if $g\in\mathcal{AC}_{1}[0,T]$.
Note that the maximization problem
\begin{equation*}
\sup_{x}\left\{xg'-(-\beta x+e^{\alpha x}-1)g\right\}
\end{equation*}
has its maximum achieved at
$x=\frac{1}{\alpha}\log\left(\frac{\beta+\frac{g'}{g}}{\alpha}\right)$,
provided that $g'\geq-\beta g$. Otherwise, the maximum is $+\infty$.
Therefore, we conclude that
\begin{equation*}
I_Z(g)=\int_{0}^{T}\frac{\beta g(t)+g'(t)}{\alpha}\log\left(\frac{\beta g(t)+g'(t)}{\alpha g(t)}\right)
-\left(\frac{\beta g(t)+g'(t)}{\alpha}-g(t)\right)dt,
\end{equation*}
for any $g\in\mathcal{AC}_{1}[0,T]$ and $g'\geq-\beta g$ and $I_{Z}(g)=+\infty$ otherwise.

\textbf{Step 4}.
Finally let us show that the rate function $I_{Z}(g)$ is good.
That is, we need to show that for any fixed $m>0$, the level set
\begin{equation}
K_{m}:=\left\{g\in\mathcal{AC}_{1}[0,T]:I_{Z}(g)\leq m\right\}
\end{equation}
is compact.

Since $Z_{t}\geq Z_{0}e^{-\beta t}$, we have $g(t)\geq g(0)e^{-\beta t}=e^{-\beta t}$
for any $t$. Therefore, for any $g\in K_{m}$,
\begin{equation}
e^{-\beta T}\int_{0}^{T}\Lambda^{\ast}\left(\frac{\beta g(t)+g'(t)}{\alpha g(t)}\right)dt\leq m,
\end{equation}
where $\Lambda^{\ast}(x):=x\log x-x+1$ is strictly convex and non-negative.
Thus, for any $g\in K_{m}$,
\begin{equation}
\int_{0}^{T}\Lambda^{\ast}\left(\frac{\beta}{\alpha}+\frac{1}{\alpha}\frac{g'(t)}{g(t)}\right)dt
\leq me^{\beta T}.
\end{equation}
Let us define $f(t)=\frac{\beta}{\alpha}t+\frac{1}{\alpha}\log g(t)$. Then
$f(0)=0$ and $f'(t)=\frac{\beta}{\alpha}+\frac{1}{\alpha}\frac{g'(t)}{g(t)}$.
From the proof that the rate function for Mogulskii's theorem is good, see e.g. Page 183 in Dembo and Zeitouni \cite{Dembo},
it follows that the set
\begin{equation}
\left\{f\in\mathcal{AC}_{0}[0,T]:\int_{0}^{T}\Lambda^{\ast}\left(f'(t)\right)dt
\leq me^{\alpha T}\right\}
\end{equation}
is a bounded set of equicontinuous functions. Since $g(t)=e^{\alpha f(t)-\beta t}$,
it follows that the set $K_{m}$ is a bounded set of equicontinuous functions.
By Arzel\`{a}-Ascoli theorem, the set $K_{m}$ is compact. Hence, $I_{Z}(g)$ is a good rate function.
The proof is complete.
\end{proof}


\begin{proof}[Proof of Theorem~\ref{thm:2}]
We apply Theorem~\ref{thm:1} and the contraction principle. One then readily obtains from \eqref{eq:Z-N} that $\mathbb{P}\left(\left\{\frac{1}{n}N_{t},0\leq t\leq T\right\}\in\cdot\right)$
satisfies a large deviation principle with the good rate function
\begin{equation}\label{eq:Ih}
I_{N}(h)=
\inf_{h(t)=\frac{g(t)-1}{\alpha}+\frac{\beta}{\alpha}\int_{0}^{t}g(s)ds, 0\leq t\leq T}I_Z(g).
\end{equation}
Observe that differentiating the integral equation
$h(t)=\frac{g(t)-1}{\alpha}+\frac{\beta}{\alpha}\int_{0}^{t}g(s)ds$, we get
\begin{equation*}
h'(t)=\frac{1}{\alpha}g'(t)+\frac{\beta}{\alpha}g(t),
\end{equation*}
which is a first-order linear ODE for $g(t)$ with initial condition $g(0)=1$.
Thus, we can solve this ODE and get
\begin{equation*}
g(t)=e^{-\beta t}+e^{-\beta t}\int_{0}^{t}\alpha e^{\beta s}h'(s)ds.
\end{equation*}
Hence, we infer from \eqref{eq:Ih} and the expression of $I_Z(g)$ in \eqref{eq:Ig} that
\begin{align*}
I_{N}(h)
&=\int_{0}^{T}h'(t)\log\frac{h'(t)}{g(t)}-(h'(t)-g(t))dt
\\
&=\int_{0}^{T}h'(t)\log\left(\frac{h'(t)}{e^{-\beta t}+e^{-\beta t}\int_{0}^{t}\alpha e^{\beta s}h'(s)ds}\right)
\nonumber
\\
&\qquad\qquad\qquad
-\left(h'(t)-e^{-\beta t}-e^{-\beta t}\int_{0}^{t}\alpha e^{\beta s}h'(s)ds\right)dt. \nonumber
\end{align*}
Using this sample path large deviations result and applying the contraction principle, we can also obtain that, $\mathbb{P}(N_T/n\in\cdot)$ satisfies
a scalar large deviation principle on $\mathbb{R}^{+}$ with the good rate function
\begin{eqnarray} \label{eq:H}
H(x;T)&=&\inf_{h:h(T)=x}I_{N}(h).
\end{eqnarray}

Next, we prove that the rate function $H$ in \eqref{eq:H} can be equivalently given by \eqref{INxT2}. Recall the moment generating function of $N_t$ in \eqref{eq:v},
\begin{eqnarray*}
\mathbb{E}[e^{\theta N_{t}}|Z_{0}=n] = \exp \left\{ \left(C\left(t; \frac{\theta}{\alpha}\right) -\frac{\theta}{\alpha}\right) n
+ D\left(t; \frac{\theta}{\alpha}\right) \right\},
\end{eqnarray*}
where
\begin{equation*}
C'\left(t;\frac{\theta}{\alpha}\right)=-\beta C\left(t;\frac{\theta}{\alpha}\right)
+e^{\alpha C\left(t;\frac{\theta}{\alpha}\right)}-1+\frac{\beta\theta}{\alpha},
\qquad
C\left(0;\frac{\theta}{\alpha}\right)=\frac{\theta}{\alpha}.
\end{equation*}

Let us first consider the critical and super-critical case, that is, $\alpha\geq\beta$.
When we have $\alpha\geq\beta$, for any $C>0$ and $\theta>0$, $-\beta C+e^{\alpha C}-1+\frac{\beta\theta}{\alpha}>0$ and
thus $C(t;\frac{\theta}{\alpha})$ is increasing in $t$. It is clear that for any $\theta>0$,
$\int_{\frac{\theta}{\alpha}}^{\infty}\frac{dC}{-\beta C+e^{\alpha C}-1+\frac{\beta\theta}{\alpha}}<\infty$.
On the other hand, it is easy to see that $\int_{0}^{\infty}\frac{dC}{-\beta C+e^{\alpha C}-1}=\infty$.
Therefore, for any fixed $T>0$, there exists a unique positive value $\theta_{d}(T)$ such that
\begin{equation}\label{thetacTII}
\int_{\frac{\theta_{d}(T)}{\alpha}}^{\infty}\frac{dC}{-\beta C+e^{\alpha C}-1+\frac{\beta\theta_{d}(T)}{\alpha}}=T.
\end{equation}
Hence, we conclude that for any fixed $T>0$, for any $0<\theta<\theta_{d}(T)$,
$C(T;\frac{\theta}{\alpha})$ is the unique positive value greater than $\frac{\theta}{\alpha}$, that satisfies the equation:
\begin{equation}\label{CTtheta}
\int_{\frac{\theta}{\alpha}}^{C(T;\frac{\theta}{\alpha})}\frac{dC}{-\beta C+e^{\alpha C}-1+\frac{\beta\theta}{\alpha}}=T.
\end{equation}
The case for $\theta\leq 0$ is similar.
Also, it is easy to see that for $\theta<\theta_{d}(T)$, $C(t;\frac{\theta}{\alpha})$ is continuous and finite in $t$, and
\begin{equation*}
D\left(T;\frac{\theta}{\alpha}\right)=\mu\int_{0}^{T}(e^{\alpha C(t;\frac{\theta}{\alpha})}-1)dt
\end{equation*}
is finite.
Therefore, for $\theta<\theta_{d}(T)$,
\begin{equation*}
\lim_{n\rightarrow\infty}\frac{1}{n}\log
\mathbb{E}[e^{\theta N_{T}}]=C\left(T;\frac{\theta}{\alpha}\right)-\frac{\theta}{\alpha}.
\end{equation*}
When $\theta\geq\theta_{d}(T)$, this limit is $\infty$.
By differentiating the equation \eqref{CTtheta} with respect to $\theta$, we get
\begin{align}\label{CTthetaDerivative}
&-\frac{1}{e^{\theta}-1}
-\frac{\beta}{\alpha}
\int_{\frac{\theta}{\alpha}}^{C(T;\frac{\theta}{\alpha})}\frac{dC}{(-\beta C+e^{\alpha C}-1+\frac{\beta\theta}{\alpha})^{2}}
\\
&\qquad
+\frac{1}{-\beta C(T;\frac{\theta}{\alpha})+e^{\alpha C(T;\frac{\theta}{\alpha})}-1+\frac{\beta\theta}{\alpha}}
\frac{d}{d\theta}C\left(T;\frac{\theta}{\alpha}\right)=0.
\nonumber
\end{align}
It is clear from the equation \eqref{thetacTII} and \eqref{CTtheta} that as $\theta\rightarrow\theta_{d}(T)$,
we have $C(T;\frac{\theta}{\alpha})\rightarrow\infty$. Therefore, from \eqref{CTthetaDerivative}, we get
\begin{align*}
\frac{\partial}{\partial\theta}C
\left(T;\frac{\theta}{\alpha}\right)
&=\left(-\beta C\left(T;\frac{\theta}{\alpha}\right)+e^{\alpha C(T;\frac{\theta}{\alpha})}-1+\frac{\beta\theta}{\alpha}\right)
\\
&\qquad
\cdot\left(\frac{1}{e^{\theta}-1}
+\frac{\beta}{\alpha}
\int_{\frac{\theta}{\alpha}}^{C(T;\frac{\theta}{\alpha})}\frac{dC}{(-\beta C+e^{\alpha C}-1+\frac{\beta\theta}{\alpha})^{2}}\right)\rightarrow\infty,
\nonumber
\end{align*}
as $\theta\rightarrow\theta_{d}(T)$.
Hence, we verified the essential smoothness condition. By G\"{a}rtner-Ellis theorem, we get
the desired result. The proof for the sub--critical case is similar and is omitted here.
\end{proof}


\section*{Acknowledgements}
The authors are grateful to the editor and an anonymous referee
for their helpful suggestions that greatly improved the quality of the paper.
Xuefeng Gao acknowledges support from Hong Kong RGC ECS Grant 2191081
and CUHK Direct Grants for Research with project codes 4055035 and 4055054.
Lingjiong Zhu is grateful to the support from NSF Grant DMS-1613164.
We thank Xiang Zhou for help with the figures.


\appendix

\section{Additional proofs}

We prove results in Sections~\ref{sec:path} and \ref{sec:largetime} in this appendix. For notational convenience, unless specified explicitly, we use $Z$ and $N$ for $Z^{n}$ and $N^{n}$ when $Z_{0}=n$. We also use $\mathbb{E}[\cdot]$ to denote the conditional expectation $\mathbb{E}[\cdot|Z_0=n]$, and $\mathbb{P}(\cdot)$ for the conditional probability $\mathbb{P}(\cdot |Z_0=n)$.

\subsection{Proofs of Propositions~\ref{prop:1} and \ref{prop:2}}
\begin{proof} [Proof of Proposition~\ref{prop:1}]
Let
\begin{equation*}
L(g,g'):=\frac{\beta g+g'}{\alpha}\log\left(\frac{\beta g+g'}{\alpha g}\right)
-\left(\frac{\beta g+g'}{\alpha}-g\right).
\end{equation*}
Then the variational problem \eqref{JxT1} becomes
\begin{equation*}
\inf_{g(0)=1,g(T)=x}\int_{0}^{T}L(g(t),g'(t))dt.
\end{equation*}
Applying the Euler-Lagrange equation $\frac{\partial L}{\partial g}-\frac{d}{dt}\frac{\partial L}{\partial g'}=0$, we deduce that
the optimal sample path $g_*$ satisfies the following equation:
\begin{equation}\label{Euler}
\frac{\beta}{\alpha}\log\left(\frac{\beta g+g'}{\alpha g}\right)
-\frac{\beta g+g'}{\alpha g}+1
-\frac{d}{dt}\left[\frac{1}{\alpha}\log\left(\frac{\beta g+g'}{\alpha g}\right)\right]=0.
\end{equation}
Define
\begin{eqnarray}\label{eq:qt}
q(t)=\frac{1}{\alpha}\log\left(\frac{\beta g_*(t)+g_*'(t)}{\alpha g_*(t)}\right).
\end{eqnarray}
Then the Equation \eqref{Euler} reduces to
\begin{equation}\label{eq:ODE-qt}
\frac{d}{dt}q(t)=\beta q(t)-(e^{\alpha q(t)}-1).
\end{equation}
Set $q(T) = \theta$, then we obtain from \eqref{eq:ODE-qt}, \eqref{eq:A-eq} and the uniqueness of ODE solutions that
\[q(t)= A(T-t; \theta), \quad \text{for $t \in [0, T]$}.\]
Note from \eqref{eq:qt} and $g_*(0)=1$, we have
\begin{eqnarray*}
g_*(t) = \exp\left( \int_{0}^t \alpha e^{\alpha q(s)}ds - \beta t \right).
\end{eqnarray*}
So what is remaining is to find the parameter $\theta$ such that $g_*(T) =x.$ We claim that the correct parameter $\theta$ is simply $\theta_*$, the maximizer to the optimization problem \eqref{JxT2}. To see this, we define $\gamma(t):= \frac{\partial}{ \partial \theta} A(t; \theta)$.
That is, $\gamma$ is the derivative of $A$ with respect to the initial condition. Then it follows from \eqref{eq:A-eq} and \cite[Chapter~V]{Hartman} that
\begin{eqnarray*}
\frac{d}{dt} \gamma(t)& = &(-\beta + \alpha e^{\alpha  A(t; \theta)}) \cdot \gamma(t),
\\
\gamma(0) &=&1,
\end{eqnarray*}
which immediately yields
\begin{eqnarray*}
\gamma(T) = \exp \left(\int_{0}^T (-\beta + \alpha e^{\alpha  A(t; \theta)}) dt \right) =g_*(T) =x.
\end{eqnarray*}
Now from \eqref{JxT2}, it is clear that
the optimal $\theta_*$ satisfies
\[x = \frac{\partial}{ \partial \theta} A(T; \theta)|_{\theta = \theta_*}.\]
Therefore, when $\theta = \theta_*,$ we have $g_*(T)=x$, and
the path $g_*$ solves the Euler-Lagrange equation \eqref{Euler}, which is a necessary condition for optimality.

It remains to check $g_*$ given by \eqref{eq:gstar} is indeed the optimal sample path for the variational problem \eqref{JxT1}.
It suffices to note that
\begin{align}
I_Z(g_*)&= \int_{0}^{T}L(g_*(t),g_*'(t))dt \nonumber\\
&=\int_{0}^{T}\left\{q(t)g_*'(t)
-(-\beta q(t)+e^{\alpha q(t)}-1)g_*(t)\right\}dt \nonumber
\\
&=\int_{0}^{T}\left\{q(t)g_*'(t)+q'(t)g_*(t)\right\}dt
\nonumber
\\
&=q(T) x-q(0)
\nonumber
\\
&=\theta_* \cdot x - A(T; \theta_*).
\nonumber \\
&=\sup_{\theta\in\mathbb{R}}\left\{\theta x-A(T;\theta)\right\},  \nonumber
\end{align}
where we have used the fact that $q(t)= A(T-t; \theta_*)$, {for $t \in [0, T]$}. Therefore, $g_*$ is indeed the optimal sample path. The proof is complete.
\end{proof}

\begin{proof} [Proof of Proposition~\ref{prop:2}]
Let us recall that
\begin{equation*}
H(x;T)=\inf_{h(0)=0,h(T)=x}I_{N}(h),
\end{equation*}
where $I_N$ is given in \eqref{INRate}.
By considering $f(t)=\int_{0}^{t}\alpha e^{\beta s}h'(s)ds$, we get
\begin{eqnarray*}
\lefteqn{H(x;T)}\\
&=\inf_{f(0)=0,\int_{0}^{T}\frac{f'(t)}{\alpha e^{\beta t}}dt=x}
\int_{0}^{T}\frac{1}{\alpha}e^{-\beta t}\left[f'(t)\log\left(\frac{f'(t)}{\alpha+\alpha f(t)}\right)
-f'(t)+\alpha+\alpha f(t)\right]dt.
\end{eqnarray*}
This is a constrained optimization problem, so we introduce the Lagrange multiplier $\lambda$ and define
\begin{equation*}
L(t,f,f')=\frac{1}{\alpha}e^{-\beta t}\left[f'(t)\log\left(\frac{f'(t)}{\alpha+\alpha f(t)}\right)
-f'(t)+\alpha +\alpha f(t)\right]- \lambda\frac{f'(t)}{\alpha e^{\beta t}}.
\end{equation*}
We consider the modified problem
\begin{equation}\label{eq:modified}
\inf_{f(0)=0,\int_{0}^{T}\frac{f'(t)}{\alpha e^{\beta t}}dt=x} \int_{0}^{T} L(t,f(t),f'(t)) dt.
\end{equation}
The Euler-Lagrange equation $\frac{\partial L}{\partial f}-\frac{d}{dt}\frac{\partial L}{\partial f'}=0$ yields that the optimal $f_*$ for \eqref{eq:modified} satisfies
\begin{equation}\label{euler2}
\frac{1}{\alpha}e^{-\beta t}\left[\frac{-f_*'(t)}{1+f_*(t)}+\alpha\right]
-\frac{d}{dt}\frac{1}{\alpha}e^{-\beta t}\left[\log\left(\frac{f_*'(t)}{\alpha+\alpha f_*(t)}\right)-\lambda\right]=0.
\end{equation}
In addition, $f_*$ satisfies the following transversality condition:
\begin{eqnarray} \label{eq:boundary}
0=\frac{\partial L}{\partial f'} \bigg|_{t=T} = \frac{1}{\alpha}e^{-\beta T}\left[\log\left(\frac{f_*'(T)}{\alpha+\alpha f_*(T)}\right)-\lambda\right].
\end{eqnarray}
Let us define
\begin{equation}\label{eq:pt}
p(t)= \frac{1}{\alpha} \log\left(\frac{f_*'(t)}{\alpha+\alpha f_*(t)} \right).
\end{equation}
Then, the Equations~\eqref{euler2} and \eqref{eq:boundary} become
\begin{eqnarray*}
\frac{dp(t)}{dt}&=&\beta p(t) -e^{\alpha p(t)} + 1 - \frac{\beta \lambda}{\alpha},\\
p(T)& =& \frac{\lambda}{\alpha}.
\end{eqnarray*}
Hence, $p$ solves a first order ODE with terminal constraint. Comparing with \eqref{eq:C1} and \eqref{eq:C2}, we infer from the uniqueness of solutions of such ODEs that
\begin{eqnarray*}
p(t)= C\left(T-t; \frac{\lambda}{\alpha}\right).
\end{eqnarray*}
Note that we can deduce from \eqref{eq:pt} that
\begin{equation*}\label{fstarp}
f_*(t)=e^{\alpha\int_{0}^{t}e^{\alpha p(s)}ds}-1.
\end{equation*}
Recall that $f_*'(t)=\alpha e^{\beta t}h_*'(t)$ and $h_*(0)=0$. Thus, we get
\begin{equation}\label{eq:h}
h_*(t)=\int_{0}^{t}\frac{f_*'(s)}{\alpha e^{\beta s}}ds=\int_{0}^{t}e^{\alpha p(s)} \cdot e^{\alpha\int_{0}^{s}e^{\alpha p(u)}du-\beta s}ds,
\end{equation}
which depends on $\lambda$.

Next, we claim that, to satisfy the constraint $h_*(T)=x$, the correct Lagrange multiplier $\lambda$ is simply $\hat \theta_*$, the maximizer to the optimization problem \eqref{INxT2}.
To see this, we set
\begin{eqnarray*}
w_{\theta}(t) = C\left(t; \frac{\theta}{\alpha}\right) - \frac{\theta}{\alpha}, \quad \text{and} \quad r(t) = \frac{\partial}{ \partial \theta} w_{\theta}(t).
\end{eqnarray*}
One readily checks from \eqref{eq:C1} and \eqref{eq:C2} that $r$ solves the ODE
\begin{eqnarray*}
\frac{d}{dt} r(t) &=& \left( -\beta + \alpha \cdot e^{\alpha C(t; \frac{\theta}{\alpha})} \right) r(t)
+ \exp\left(\alpha C\left(t; \frac{\theta}{\alpha}\right)\right) ,\\
r(0)&=&0,
\end{eqnarray*}
which implies
\begin{eqnarray*}
r(T) = \int_{0}^T   e^{\alpha C(t; \frac{\theta}{\alpha})}\cdot \exp\left( \int_{t}^T \left( \alpha e^{\alpha C(s; \frac{\theta}{\alpha})} - \beta\right) ds   \right) dt.
\end{eqnarray*}
Since $\hat \theta_*$ is the maximizer to the optimization problem \eqref{INxT2}, we deduce that
\begin{eqnarray}\label{eq:x-rT}
x = r(T)|_{\theta=\hat \theta_* }= \int_{0}^T   e^{\alpha C(T-u; \frac{\hat \theta_*}{\alpha})}\cdot \exp\left( \int_{0}^u \left( \alpha e^{\alpha C(T-v; \frac{\hat \theta_*}{\alpha})} - \beta\right) dv   \right) du,
\end{eqnarray}
where we have applied the change of variable formula. Then it is clear from \eqref{eq:h} and \eqref{eq:x-rT} that when $p(t)= C(T-t; \frac{\hat \theta_*}{\alpha})$, we have $h_*(T)=x$.

Finally, we verify that $h_*$ given by \eqref{eq:hstar} is indeed optimal for the variational problem \eqref{INxT}.
It suffices to note that
\begin{align}
I_{N}(h_{*})
&=\hat{\theta}_{*}\cdot x+\int_{0}^{T}\frac{e^{-\beta t}}{\alpha}
\left[f'_{*}(t)\log\left(\frac{f'_{*}(t)}{\alpha+\alpha f_{*}(t)}\right)
-f'_{*}(t)+\alpha+\alpha f_{*}(t)-\hat{\theta}_{*}f'_{*}(t)\right]dt
\nonumber \\
&=\hat{\theta}_{*}\cdot x+
\int_{0}^{T}\frac{e^{-\beta t}}{\alpha}f'_{*}(t)\left[\log\left(\frac{f'_{*}(t)}{\alpha+\alpha f_{*}(t)}\right)-\hat{\theta}_{*}\right]dt
\nonumber
\\
&\qquad\qquad\qquad\qquad\qquad
+\int_{0}^{T}\frac{e^{-\beta t}}{\alpha}\left[\frac{-f'_{*}(t)}{1+f_{*}(t)}+\alpha\right](1+f_{*}(t))dt
\nonumber
\\
&=\hat{\theta}_{*}\cdot x+
\int_{0}^{T}\frac{d}{dt}(1+f_{*}(t))\frac{e^{-\beta t}}{\alpha}\left[\log\left(\frac{f'_{*}(t)}{\alpha+\alpha f_{*}(t)}\right)-\hat{\theta}_{*}\right]dt
\nonumber
\\
&\qquad\qquad\qquad\qquad\qquad
+\int_{0}^{T}\frac{d}{dt}\left(\frac{e^{-\beta t}}{\alpha}\left[\log\left(\frac{f'_{*}(t)}{\alpha+\alpha f_{*}(t)}\right)-\hat{\theta}_{*}\right]\right)(1+f_{*}(t))dt
\nonumber
\\
&=\hat{\theta}_{*} \cdot x
+(1+f_{*}(T))\frac{e^{-\beta T}}{\alpha}\left[\log\left(\frac{f'_{*}(T)}{\alpha+\alpha f_{*}(T)}\right)-\hat{\theta}_{*}\right]
\nonumber
\\
&\qquad\qquad\qquad\qquad\qquad
-(1+f_{*}(0))\frac{1}{\alpha}\left[\log\left(\frac{f'_{*}(0)}{\alpha+\alpha f_{*}(0)}\right)-\hat{\theta}_{*}\right],
\nonumber
\end{align}
where the first equality is due to $f_*'(t)=\alpha e^{\beta t}h_*'(t)$, $h_*(0)=0$ and $h_*(T)=x$, the third equality is due to \eqref{euler2}, and other equalities follow from direct computation. Now by \eqref{eq:boundary}, \eqref{eq:pt} and $f_{*}(0)=0$, and the fact that $p(t)=C(T-t;\frac{\hat{\theta}_{*}}{\alpha})$ for any $0\leq t\leq T$, we obtain
\begin{align}
I_{N}(h_{*})
&=\hat{\theta}_{*} \cdot x - p(0)+\frac{\hat{\theta}_{*}}{\alpha}
\nonumber
\\
&=\hat{\theta}_{*} \cdot x - C\left(T; \frac{\hat{\theta}_{*}}{\alpha}\right)+\frac{\hat{\theta}_{*}}{\alpha}
\nonumber
\\
&=\sup_{\theta\in\mathbb{R}}\left\{\theta x - C\left(T; \frac{\theta}{\alpha}\right)+\frac{\theta}{\alpha}\right\},
\nonumber
\end{align}
Therefore, $h_*$ is indeed the optimal sample path. The proof is complete.
\end{proof}


\subsection{Proofs of results in Section~\ref{sec:critical} } \label{sec:proof-critical}
\begin{proof}[Proof of Theorem~\ref{thm:critical}]
We want to apply G\"{a}rtner-Ellis theorem
to obtain the large deviations principle. Since the moment generating functions of $Z_t$ and $N_t$ involves nonlinear ODEs, the key idea in the proof is
to use Gronwall's inequality for nonlinear ODEs to obtain estimates for the ODE solutions. Below we prove part (i) and (ii) separately.

(i) We first prove part (i).
Suppose that we can show
\begin{eqnarray}\label{eq:z-lim-crit}
\lim_{n\rightarrow\infty}\frac{t_{n}}{n}\log\mathbb{E}\left[e^{\frac{\theta}{t_{n}}Z_{t_{n}T}}\right] =
\begin{cases}
\frac{\theta}{1-\frac{1}{2}\alpha^{2}\theta T} \quad \text{for any $\theta<\frac{2}{\alpha^{2}T}$, }\\
\infty \quad\qquad\quad \text{otherwise.}
\end{cases}
\end{eqnarray}
It is easy to check that $\frac{\theta}{1-\frac{1}{2}\alpha^{2}\theta T}$ is differentiable in $\theta$
for any $\theta<\frac{2}{\alpha^{2}T}$, and $\frac{\partial}{\partial\theta}\left[\frac{\theta}{1-\frac{1}{2}\alpha^{2}\theta T}\right]
=\frac{1}{(1-\frac{1}{2}\alpha^{2}\theta T)^{2}}\rightarrow\infty$ as $\theta\rightarrow\frac{2}{\alpha^{2}T}$. Thus,
we verified the essential smoothness condition.
By G\"{a}rtner-Ellis Theorem, $\mathbb{P}(\frac{Z_{t_{n}T}}{n}\in\cdot)$ satisfies
a large deviation principle with the speed $\frac{n}{t_{n}}$ and the good rate function
\begin{equation*}
\hat{I}_{Z}(x)=\sup_{\theta<\frac{2}{\alpha^{2}T}}\left\{\theta x-\frac{\theta}{1-\frac{1}{2}\alpha^{2}\theta T}\right\}
=\frac{2(\sqrt{x}-1)^{2}}{\alpha^{2}T},
\end{equation*}
if $x\geq 0$ and $+\infty$ otherwise.

Therefore, it suffices to prove \eqref{eq:z-lim-crit}. We focus on the case $\theta \ne 0$ since the proof is trivial for the case $\theta =0$.

From the moment generating function of $Z_t$ in \eqref{eq:u}, one readily obtains
\begin{eqnarray*}
\frac{t_{n}}{n}\log\mathbb{E}\left[e^{\frac{\theta}{t_{n}}Z_{t_{n}T}}\right] = t_n \cdot A\left(t_{n}T;\frac{\theta}{t_{n}}\right) + \frac{t_{n}}{n} \cdot B\left(t_{n}T;\frac{\theta}{t_{n}}\right),
\end{eqnarray*}
where $A$ are $B$ are solutions to the ODEs in \eqref{eq:ODE-A} and \eqref{eq:ODE-B}. Here the initial conditions are given by $A\left(0; \frac{\theta}{t_{n}}\right) = \frac{\theta}{t_{n}} $ and $B\left(0;\frac{\theta}{t_{n}}\right)=0$.
So in order to show \eqref{eq:z-lim-crit}, it suffices to show
\begin{eqnarray} \label{eq:lim-A}
\lim_{n \rightarrow \infty} t_n \cdot A\left(t_{n}T;\frac{\theta}{t_{n}}\right) =
\begin{cases}
\frac{\theta}{1-\frac{1}{2}\alpha^{2}\theta T}, & \text{for $\theta< \frac{2}{\alpha ^2 T}$},\\
+\infty, & \text{for $\theta \ge \frac{2}{\alpha ^2 T}$,}
\end{cases}
\end{eqnarray}
and
\begin{eqnarray} \label{eq:lim-B}
\begin{cases}
\lim_{n \rightarrow \infty} \frac{t_{n}}{n} \cdot B\left(t_{n}T;\frac{\theta}{t_{n}}\right) =
0, & \text{for $\theta< \frac{2}{\alpha ^2 T}$},\\
\liminf_{n \rightarrow \infty} \frac{t_{n}}{n} \cdot B\left(t_{n}T;\frac{\theta}{t_{n}}\right) \ge
0, & \text{for $\theta \ge \frac{2}{\alpha ^2 T}$.}
\end{cases}
\end{eqnarray}

We first prove \eqref{eq:lim-A}. The idea is to use Gronwall's inequality for nonlinear ODEs to obtain estimates for $A$.
Write $g(x) = e^{\alpha x} - \alpha x -1$. Then the ODE in \eqref{eq:ODE-A} becomes $A'(t) = g(A(t))$ in the critical case $\alpha=\beta$.
Given small $\epsilon, \eta>0$, there exists some $\delta>0$ such that $  (\frac{\alpha^2}{2} - \epsilon) x^2 \le g(x) \le  (\frac{\alpha^2}{2} + \epsilon) x^2 $ and $|g(x)| \le \eta |x|$
when $|x| \le \delta$. If we write
\[c_n= \sup \left\{t \ge 0: \left|A\left(s;\frac{\theta}{t_{n}}\right)\right| \le \delta, \quad\text{for all $s \le t$}\right\},\] then we obtain for $n$ large,
\begin{eqnarray} \label{eq:gronwall-1}
\left(\frac{\alpha^2}{2} - \epsilon\right) A^2\left(t;\frac{\theta}{t_{n}}\right) \le A'\left(t;\frac{\theta}{t_{n}}\right) \le \left(\frac{\alpha^2}{2} + \epsilon\right) A^2\left(t;\frac{\theta}{t_{n}}\right), \quad \text{for all $t \in [0, c_n]$}.
\end{eqnarray}
Note that the solution to the ODE
\begin{eqnarray*} \label{eq:y}
y'(t) &=& \left(\frac{\alpha^2}{2} + \epsilon\right) y^2(t), \\
y(0) &=& y_{0},
\end{eqnarray*}
is given by
\begin{equation} \label{eq:y-sol}
y(t) = \left( \frac{1}{y_0} - \left(\frac{\alpha^2}{2} + \epsilon\right) t \right)^{-1},
\end{equation}
which is clearly an non-decreasing function when $y$ is properly defined. We next discuss three cases to prove \eqref{eq:lim-A}.

\textbf{Case 1: $\theta<0$.}

When $\alpha = \beta$, it is clear from \eqref{eq:ODE-A} that $A$ is non-decreasing in $t$. So for $n$ large such that $0>A\left(0; \frac{\theta}{t_{n}}\right) = \frac{\theta}{t_{n}} > -\delta$, one readily checks that $0>A(t; \frac{\theta}{t_{n}})> -\delta$ for all $t \ge 0$.
That is, $c_n=+\infty$. Hence, we deduce from \eqref{eq:gronwall-1}--\eqref{eq:y-sol} and Gronwall's inequality for nonlinear ODEs that
\begin{eqnarray*}
\left( \frac{t_{n}}{\theta} - \left(\frac{\alpha^2}{2} - \epsilon\right) t \right)^{-1} \le A\left(t; \frac{\theta}{t_{n}}\right) \le \left( \frac{t_{n}}{\theta}  - \left(\frac{\alpha^2}{2} + \epsilon\right) t \right)^{-1} , \quad \text{for all $t \ge 0$}.
\end{eqnarray*}
This implies
\begin{align}
\left( \frac{1}{\theta}  - \left(\frac{\alpha^2}{2} - \epsilon\right) T \right)^{-1}
&\le \liminf_{n \rightarrow \infty} t_n \cdot A\left(t_{n}T;\frac{\theta}{t_{n}}\right) \nonumber
\\
&\le \limsup_{n \rightarrow \infty} t_n \cdot A\left(t_{n}T;\frac{\theta}{t_{n}}\right) \le \left( \frac{1}{\theta}  - \left(\frac{\alpha^2}{2} + \epsilon\right) T \right)^{-1} . \label{eq:liminfsupA}
\end{align}
Letting $\epsilon \rightarrow 0$, we obtain that for any $\theta< 0$,
\begin{eqnarray}\label{eq:limit-A}
\lim_{n \rightarrow \infty} t_n  A\left(t_{n}T;\frac{\theta}{t_{n}}\right) = \left( \frac{1}{\theta}  - \frac{\alpha^2}{2} T \right)^{-1} = \frac{\theta}{1-\frac{1}{2}\alpha^{2}\theta T} .
\end{eqnarray}

\textbf{Case 2: $0<\theta<\frac{2}{\alpha ^2 T}$.}

In this case, we have $\theta T ( \frac{\alpha^2}{2} + \epsilon ) < 1$ for $\epsilon$ small enough. This implies that given $y(0)=A(0;\frac{\theta}{t_{n}})= \frac{\theta}{t_n}$, the ODE solution $y$ in \eqref{eq:y-sol} is properly defined for $t= t_nT$, and its value is given by
\begin{eqnarray*}
y(t_n T) = \left( \frac{t_n}{\theta} - \left(\frac{\alpha^2}{2} + \epsilon\right) t_nT \right)^{-1} = \frac{1}{t_n} \cdot \frac{\theta}{1- \theta T (\frac{\alpha^2}{2} + \epsilon)}.
\end{eqnarray*}
Hence from the non-decreasing property of $y$, we have $0 <y(t) < \delta$ for all $t \in [0, t_n T]$ when $t_n$ is large. This implies $t_n T \le c_n$. Following the proof of Case 1, we can deduce from \eqref{eq:gronwall-1}--\eqref{eq:y-sol} and Gronwall's inequality for nonlinear ODEs that \eqref{eq:limit-A} holds for $0<\theta<\frac{2}{\alpha ^2 T}$ as well.

\textbf{Case 3: $\theta \ge \frac{2}{\alpha ^2 T}$.}

When $\theta \ge \frac{2}{\alpha ^2 T}$,  we prove \eqref{eq:lim-A} by contradiction. Suppose
\[\liminf_{n \rightarrow \infty} t_n \cdot A\left(t_{n}T;\frac{\theta}{t_{n}}\right) = M< +\infty. \]
Then there is a subsequence $\{n_k\}$ such that
\begin{eqnarray} \label{eq:contra-1}
t_{n_k}  \cdot A\left(t_{n_k}T;\frac{\theta}{t_{n_k}}\right) \le 2M, \quad \text{for $n_k$ large}.
\end{eqnarray}
This implies when $n_k$ is large, we have $0 < A\left(t_{n_k}T;\frac{\theta}{t_{n_k}}\right) \le \delta$, which further implies $c_{n_k} \ge t_{n_k}T.$ Similar as before, we can use \eqref{eq:gronwall-1} and apply
Gronwall's inequality for nonlinear ODEs and obtain that
\begin{eqnarray*}
A\left(t_{n_k}T;\frac{\theta}{t_{n_k}}\right) \ge \frac{1}{t_{n_k}} \cdot \frac{\theta}{1- \theta T (\frac{\alpha^2}{2} - \epsilon)} \ge \frac{1}{t_{n_k}} \cdot \frac{1}{\epsilon T},
\end{eqnarray*}
where the last inequality is due to the fact that $\theta \ge \frac{2}{\alpha ^2 T}$. However, this is a contradiction with \eqref{eq:contra-1} since we can choose $\epsilon$ arbitrarily small. Hence, for $\theta \ge \frac{2}{\alpha ^2 T}$, we have
\begin{eqnarray*}
\lim_{n \rightarrow \infty} t_n  A\left(t_{n}T;\frac{\theta}{t_{n}}\right) = +\infty .
\end{eqnarray*}
Therefore, we have proved \eqref{eq:lim-A}.

We next prove \eqref{eq:lim-B}. When $\theta < \frac{2}{\alpha ^2 T}$, we have shown that $t_n T \le c_n$ for $n$ large and thus $|A(t; \frac{\theta}{t_n})| \le \delta$ for all $t \in [0, t_n T]$.
together with the fact that $|e^{\alpha x} - 1 -\alpha x| \le \eta |x|$ when $|x| \le \delta$, we deduce from \eqref{eq:ODE-B}
that for $\theta< \frac{2}{\alpha ^2 T}$,
\begin{eqnarray*}
\left|B\left(t_{n}T;\frac{\theta}{t_n}\right)\right| & = & \mu \left|\int_{0}^{t_{n}T} \left(e^{\alpha A\left(s;\frac{\theta}{t_n}\right)} -1 \right) ds \right|\\
&\le& \mu (\alpha + \eta) \int_{0}^{t_{n}T} \left| A\left(s;\frac{\theta}{t_n}\right) \right| ds.
\end{eqnarray*}
In conjunction with the inequality \eqref{eq:liminfsupA}, it is readily verified that for $\theta< \frac{2}{\alpha ^2 T}$,
\begin{eqnarray*}
\lim_{n \rightarrow \infty} (t_n/n) \cdot B\left(t_{n}T;\frac{\theta}{t_{n}}\right) = 0 .
\end{eqnarray*}
For $\theta \ge\frac{2}{\alpha ^2 T}$, it is clear that $A$ is always positive for all $t$. Thus we infer from \eqref{eq:ODE-B} that $B$ is nonnegative and
\begin{eqnarray*}
\liminf_{n \rightarrow \infty} (t_n/n) \cdot B\left(t_{n}T;\frac{\theta}{t_{n}}\right) \ge 0 .
\end{eqnarray*}
Therefore, we have proved \eqref{eq:lim-B} and thus \eqref{eq:z-lim-crit} holds. The proof of part (i) is complete.

(ii) We next prove part (ii). Recall the moment generating function of $N_t$ in \eqref{eq:v}. Since $Z_0=n$, we infer that
\begin{equation}\label{eq:mgf-n-scaled}
\frac{t_{n}}{n}\log\mathbb{E}\left[e^{\frac{\theta}{t_{n}^{2}}N_{t_{n}T}}\right]
= \frac{t_{n}}{n} \left( -\frac{\theta}{\alpha t_{n}^{2}}n
+ C\left(t_{n}T;\frac{\theta}{\alpha t_{n}^{2}}\right)n+D\left(t_{n}T;\frac{\theta}{\alpha t_{n}^{2}}\right)\right),
\end{equation}
where $C, D$ solve the ODEs in \eqref{eq:ODE-C} and \eqref{eq:ODE-D}
with initial condition $C(0; \frac{\theta}{\alpha t_{n}^{2}})=\frac{\theta}{\alpha t_{n}^{2}} $ and $D(0;\frac{\theta}{\alpha t_{n}^{2}})=0$.

To study its limiting behavior as $n \rightarrow \infty$, we first prove
\begin{equation} \label{eq:limit-C}
\lim_{n\rightarrow\infty} t_n \cdot C\left(t_{n}T;\frac{\theta}{\alpha t_{n}^{2}}\right)
=
\begin{cases}
\frac{\sqrt{-\theta}}{\frac{\alpha}{\sqrt{2}}}\tanh\left(\frac{-\alpha}{\sqrt{2}}\sqrt{-\theta}T\right) &\text{if $\theta\leq 0$}
\\
\frac{\sqrt{\theta}}{\frac{\alpha}{\sqrt{2}}}\tan\left(\frac{\alpha}{\sqrt{2}}\sqrt{\theta}T\right)
&\text{if $\theta>0$}
\end{cases}.
\end{equation}
We focus on $\theta \ne 0$ since the case $\theta= 0$ is trivial to prove. Similar as in the proof of part (i), we obtain for $n$ large, and {for all $t \in [0, d_n]$}
\begin{align} \label{eq:compare1}
\left(\frac{\alpha^2}{2} - \epsilon\right) C^2\left(t;\frac{\theta}{ \alpha t^2_{n}}\right) + \alpha C\left(0;\frac{\theta}{ \alpha t^2_{n}}\right)
&\le C'\left(t;\frac{\theta}{ \alpha t^2_{n}}\right)
\\
&\le \left(\frac{\alpha^2}{2} + \epsilon\right) C^2\left(t;\frac{\theta}{ \alpha t^2_{n}}\right) + \alpha C\left(0;\frac{\theta}{ \alpha t^2_{n}}\right),\nonumber
\end{align}
where
\[d_n= \sup \left\{t \ge 0: \left|C\left(s;\frac{\theta}{ \alpha t^2_{n}}\right)\right| \le \delta, \quad\text{for all $s \le t$}\right\}.\]
We again want to use Gronwall's inequality for nonlinear ODEs to obtain estimates for $C$. To this end, we first study the solution $z_k$ to the Riccati equation
\begin{eqnarray}
z'(t) &=& k z^2(t) + \alpha z_0, \label{eq:z}\\
z(0) &=& z_0 = \frac{\theta}{ \alpha t^2_{n}} \label{eq:z0},
\end{eqnarray}
where we are interested in $k=\frac{\alpha^2}{2} + \epsilon$ and $k=\frac{\alpha^2}{2} -\epsilon$.
We discuss the cases $\theta>0$ and $\theta <0$ separately.

When $\theta>0$, the solution to the Riccati equation \eqref{eq:z} and \eqref{eq:z0} is given by
\begin{eqnarray*}
z_k(t) = \frac{1}{k} \cdot \frac{ \frac{\sqrt{k\theta}}{t_n} \sin\left(\frac{\sqrt{k\theta}}{t_n} t\right) + \frac{k\theta}{\alpha t_n^2} \cos\left(\frac{\sqrt{k\theta}}{t_n} t\right) }{\cos\left(\frac{\sqrt{k\theta}}{t_n} t\right) -  \frac{\sqrt{k\theta}}{t_n} \sin\left(\frac{\sqrt{k\theta}}{t_n} t\right)}.
\end{eqnarray*}
It follows from \eqref{eq:compare1} and Gronwall's inequality for nonlinear ODEs that
\begin{eqnarray} \label{eq:compare2}
z_{\frac{\alpha^2}{2} - \epsilon}(t) \le C\left(t;\frac{\theta}{ \alpha t^2_{n}}\right)
\le z_{\frac{\alpha^2}{2} + \epsilon}(t), \quad \text{for $t \in [0, d_n]$}.
\end{eqnarray}
For $k= \frac{\alpha^2}{2} + \epsilon$ or $k= \frac{\alpha^2}{2} - \epsilon$, it is clear that $|z_k(t_nT)|= O(t_n^{-1})$ as $t_n \rightarrow \infty$, from which one can verify that $t_n T \le d_n$. Together with \eqref{eq:compare2}, we obtain
\begin{equation*}
\limsup_{n\rightarrow\infty} t_n \cdot C\left(t_{n}T;\frac{\theta}{\alpha t_{n}^{2}}\right) \le  \lim_{n\rightarrow\infty} t_n \cdot z_{\frac{\alpha^2}{2} + \epsilon}(t_{n}T)
= \frac{\sqrt{\theta}}{\sqrt{\frac{\alpha^2}{2} + \epsilon}}\tan\left(\sqrt{\left(\frac{\alpha^2}{2} + \epsilon\right)\theta}T\right).
\end{equation*}
Setting $\epsilon \rightarrow 0$, we find
\begin{equation*}
\limsup_{n\rightarrow\infty} t_n \cdot C\left(t_{n}T;\frac{\theta}{\alpha t_{n}^{2}}\right) \le \frac{\sqrt{\theta}}{\frac{\alpha}{\sqrt{2}}}\tan\left(\frac{\alpha}{\sqrt{2}}\sqrt{\theta}T\right).
\end{equation*}
A similar argument leads to
\begin{equation*}
\liminf_{n\rightarrow\infty} t_n \cdot C\left(t_{n}T;\frac{\theta}{\alpha t_{n}^{2}}\right) \ge \frac{\sqrt{\theta}}{\frac{\alpha}{\sqrt{2}}}\tan\left(\frac{\alpha}{\sqrt{2}}\sqrt{\theta}T\right).
\end{equation*}
So we have proved \eqref{eq:limit-C} when $\theta>0$.

When $\theta<0$, we can similarly solve the Riccati equation \eqref{eq:z} and \eqref{eq:z0} and obtain
\begin{eqnarray*}
z_k(t) = -\frac{1}{k} \cdot  \frac{\sqrt{k|\theta|}}{t_n} \cdot \frac{ \exp\left( \frac{\sqrt{k|\theta|}}{t_n}\right) - r \cdot\exp\left( -\frac{\sqrt{k|\theta|}}{t_n}\right) }{\exp\left( \frac{\sqrt{k|\theta|}}{t_n}\right) + r \cdot\exp\left( -\frac{\sqrt{k|\theta|}}{t_n}\right)},
\end{eqnarray*}
where $r:= \frac{1- \sqrt{\frac{{k|\theta|}}{\alpha t_n} }}{ 1 + \sqrt{\frac{{k|\theta|}}{\alpha t_n} }}$. Thus we infer from \eqref{eq:compare2} that for $k=\frac{\alpha^2}{2} + \epsilon$,
\begin{equation*}
\limsup_{n\rightarrow\infty} t_n \cdot C\left(t_{n}T;\frac{\theta}{\alpha t_{n}^{2}}\right) \le  \lim_{n\rightarrow\infty} t_n \cdot z_k(t_{n}T)
= - \frac{\sqrt{|\theta|}}{\sqrt{k}}\tanh \left(\sqrt{k|\theta|}T\right).
\end{equation*}
Setting $\epsilon \rightarrow 0$, we find for $\theta<0$,
\begin{equation*}
\limsup_{n\rightarrow\infty} t_n \cdot C\left(t_{n}T;\frac{\theta}{\alpha t_{n}^{2}}\right) \le -\frac{\sqrt{|\theta|}}{\frac{\alpha}{\sqrt{2}}}\tanh \left(\frac{\alpha}{\sqrt{2}}\sqrt{|\theta|}T\right).
\end{equation*}
A similar argument leads to
\begin{equation*}
\liminf_{n\rightarrow\infty} t_n \cdot C\left(t_{n}T;\frac{\theta}{\alpha t_{n}^{2}}\right) \ge -\frac{\sqrt{|\theta|}}{\frac{\alpha}{\sqrt{2}}}\tanh \left(\frac{\alpha}{\sqrt{2}}\sqrt{|\theta|}T\right).
\end{equation*}
So we have proved \eqref{eq:limit-C} when $\theta<0$.

We next show
\begin{equation*} \label{eq:limit-D}
\lim_{n\rightarrow\infty} (t_n/n)  \cdot D\left(t_{n}T;\frac{\theta}{\alpha t_{n}^{2}}\right) =0.
\end{equation*}
Since we have shown $|C(t; \frac{\theta}{\alpha t^2_n })| \le \delta$ for all $t \in [0, t_n T]$,
together with the fact that $|e^{\alpha x} - 1 -\alpha x| \le \eta |x|$ when $|x| \le \delta$, we deduce that
\begin{eqnarray*}
\left|D \left(t_{n}T;\frac{\theta}{\alpha t^2_n } \right)\right| & = & \mu \left|\int_{0}^{t_{n}T} \left(e^{\alpha C\left(s;\frac{\theta}{\alpha t^2_n } \right)} -1 \right) ds \right|\nonumber\\
&\le& \mu (\alpha + \eta) \int_{0}^{t_{n}T} \left| C\left(s;\frac{\theta}{\alpha t^2_n } \right) \right| ds\nonumber\\
&=& \mu (\alpha + \eta) \int_{0}^{T} \left| t_n \cdot C\left(s t_n;\frac{\theta}{\alpha t^2_n } \right) \right| ds.\label{CDBound}
\end{eqnarray*}
Given \eqref{eq:limit-C} and \eqref{CDBound}, we obtain
\begin{eqnarray*}
\limsup_{n \rightarrow \infty}\left| D\left(t_{n}T;\frac{\theta}{\alpha t^2_n }\right)\right| \le K, \quad \text{for some constant $K<\infty$}.
\end{eqnarray*}
Since $t_n/n \rightarrow 0,$ we then deduce that
\begin{eqnarray*}
\lim_{n \rightarrow \infty} (t_n/n) \cdot D\left(t_{n}T;\frac{\theta}{\alpha t^2_n }\right) = 0 .
\end{eqnarray*}
Hence, we infer from \eqref{eq:mgf-n-scaled} that
\begin{equation*}\label{eq:2}
\Lambda(\theta):=\lim_{n\rightarrow\infty}\frac{t_{n}}{n}\log\mathbb{E}\left[e^{\frac{\theta}{t_{n}^{2}}N_{t_{n}T}}\right]
=
\begin{cases}
\frac{\sqrt{-\theta}}{\frac{\alpha}{\sqrt{2}}}\tanh\left(\frac{-\alpha}{\sqrt{2}}\sqrt{-\theta}T\right) &\text{if $\theta\leq 0$}
\\
\frac{\sqrt{\theta}}{\frac{\alpha}{\sqrt{2}}}\tan\left(\frac{\alpha}{\sqrt{2}}\sqrt{\theta}T\right)
&\text{if $\theta>0$}
\end{cases}.
\end{equation*}
It is easy to verify that $\Lambda(\theta)$ is differentiable in $\theta<0$ and
\begin{equation*}
\frac{d}{d\theta}\Lambda(\theta)=\frac{1}{2\sqrt{\theta}}\frac{1}{\frac{\alpha}{\sqrt{2}}}\tan\left(\frac{\alpha}{\sqrt{2}}\sqrt{\theta}T\right)
+\frac{T}{2}\sec^{2}\left(\frac{\alpha}{\sqrt{2}}\sqrt{\theta}T\right)
\qquad\text{for $\theta<0$},
\end{equation*}
and $\Lambda(\theta)$ is differentiable in $\theta>0$ and
\begin{equation*}
\frac{d}{d\theta}\Lambda(\theta)=\frac{-1}{2\sqrt{-\theta}}\frac{1}{\frac{\alpha}{\sqrt{2}}}\tanh\left(\frac{-\alpha}{\sqrt{2}}\sqrt{-\theta}T\right)
+\frac{T}{2}\mbox{sech}^{2}\left(\frac{-\alpha}{\sqrt{2}}\sqrt{-\theta}T\right)
\qquad\text{for $\theta>0$}.
\end{equation*}
Thus, it is easy to see that
\begin{equation*}
\Lambda'(0+)=\Lambda'(0-)=T,
\end{equation*}
which implies that $\Lambda(\theta)$ is also differentiable at $\theta=0$. An application of G\"{a}rtner-Ellis Theorem yields the result.
%
\end{proof}

\subsection{Proofs of results in Section~\ref{sec:supercritical}} \label{sec:proof-super-critical}
\begin{proof}[Proof of Theorem~\ref{thm:supercritical}]
The approach is similar as in the proof of Theorem~\ref{thm:critical}: use Gronwall's inequality for nonlinear ODEs to obtain estimates, and then apply G\"{a}rtner-Ellis theorem
to obtain the large deviations principle.

(i) We first prove part (i). We claim that for $Z_0=n$ and any $\theta\in\mathbb{R}$,
\[\lim_{n\rightarrow\infty}\frac{1}{n^{T}}\log\mathbb{E}\left[e^{\frac{\theta}{n}Z_{t_{n}T}}\right] =  \lim_{n\rightarrow\infty}\frac{1}{n^{T}}\cdot  \left({nA\left(t_{n}T;\frac{\theta}{n}\right)+B\left(t_{n}T;\frac{\theta}{n}\right)} \right) = \theta.\]

When $\theta=0$, the above holds trivially. So in the following we focus on $\theta \ne 0$.

We first show that
\[ \lim_{n\rightarrow\infty}n^{1-T} \cdot A\left(t_{n}T;\frac{\theta}{n}\right) = \theta .  \]
Write $g(x) = e^{\alpha x} - \alpha x -1$.
Then given any small $\eta>0$, there exists some $\delta>0$ and $K>0$ such that $|g(x)| \le \min\{\eta |x|, K x^2\}$ when $|x| < \delta$. Recall from \eqref{eq:ODE-A} that $A$ solves the ODE:
\begin{align}
&A'\left(t;\frac{\theta}{n}\right) =(\alpha - \beta)A\left(t;\frac{\theta}{n}\right)+ g\left(A\left(t;\frac{\theta}{n}\right)\right), \label{eq:ODE-A1}
\\
&A\left(0;\frac{\theta}{n}\right)= \frac{\theta}{n}. \nonumber
\end{align}
Suppose that for $n$ large, we have $\left|A\left(t; \frac{\theta}{n}\right)\right| \le \delta$ for all $t \in [0, t_n T]$. Then we obtain
\begin{eqnarray*}
(\alpha - \beta)A\left(t;\frac{\theta}{n}\right) - K A^2\left(t;\frac{\theta}{n}\right) \le A'\left(t;\frac{\theta}{n}\right) \le (\alpha - \beta)A\left(t;\frac{\theta}{n}\right) + K A^2\left(t;\frac{\theta}{n}\right).
\end{eqnarray*}
The solution to the Bernoulli equation
\begin{eqnarray*}
y'(t) &=& (\alpha - \beta)y(t) + K y^2(t), \\
y(0)&=& A(0) = \frac{\theta}{n},
\end{eqnarray*}
is given by
\[y(t) = \left( \left(\frac{1}{y(0)} + \frac{K}{\alpha - \beta}\right) \cdot e^{(\beta - \alpha) t} - \frac{K}{\alpha - \beta} \right)^{-1}.\]
Hence we have
\[y(t_n T) = \left( \frac{n^{1-T}}{\theta} + \frac{K}{\alpha - \beta} \cdot n^{-T} - \frac{K}{\alpha - \beta} \right)^{-1}.  \]
It is clear that $y$ is a monotone function, and we thus deduce that $|y(t)| \le \delta$ for all $t \in [0, t_n T]$ when $n$ is large.
Note that $L(y):=(\alpha - \beta)y + K y^2$ is Lipshitz continuous when $|y| \le \delta$.
Then by Gronwall's inequality for nonlinear ODEs, we obtain
\[A(t) \le y(t), \quad \text{for $t \in [0, t_n T]$},\]
which further implies that
\[ A(t_n T ) \le y(t_n T) = \left( \frac{n^{1-T}}{\theta} + \frac{K}{\alpha - \beta} \cdot n^{-T} - \frac{K}{\alpha - \beta} \right)^{-1}. \]
Similarly, we can find
\[ A(t_n T ) \ge  \left( \frac{n^{1-T}}{\theta} + \frac{-K}{\alpha - \beta} \cdot n^{-T} - \frac{-K}{\alpha - \beta} \right)^{-1} . \]
Thus we get
\[ \lim_{n\rightarrow\infty}n^{1-T} \cdot A\left(t_{n}T;\frac{\theta}{n}\right) = \theta .  \]

So the only remaining step is to show for $n$ large, we have $|A(t; \frac{\theta}{n})| \le \delta$ for all $t \in [0, t_n T]$.
To this end, we first define, with a slight abuse of notation (see also Section~\ref{sec:proof-critical})
\[c_n = \sup \left\{t \ge0 : \left|A \left(s; \frac{\theta}{n}\right)\right| \le \delta, \quad \text{for all $s\le t$} \right\},\]
and note that $\left|A\left(0; \frac{\theta}{n}\right)\right| =\left| \frac{\theta}{n}\right| < \delta $ for all large $n$. 
So it suffices to show $t_n T \le c_n$.
We note from the ODE in \eqref{eq:ODE-A1} that
\[A\left(t; \frac{\theta}{n}\right) = A\left(0; \frac{\theta}{n}\right) \cdot e^{(\alpha - \beta)t}
+ \int_{0}^t e^{(\alpha - \beta) (t-s)} g\left(A\left(s; \frac{\theta}{n}\right)\right) ds. \]
Note that $\left|A\left(t; \frac{\theta}{n}\right)\right| < \delta$ for all $ t \in [0, c_n]$. Then we have $g\left(A\left(t; \frac{\theta}{n}\right)\right) \le \eta \left|A(t; \frac{\theta}{n})\right|$ for all $t \in [0,c_n]$.
Hence we obtain for $t \in [0, c_n]$
\[ e^{-(\alpha - \beta)t} \left|A\left(t; \frac{\theta}{n}\right)\right| \le \left|A\left(0; \frac{\theta}{n}\right)\right|
+ \eta \int_{0}^t e^{-(\alpha - \beta)s} \left|A\left(s; \frac{\theta}{n}\right)\right| ds.  \]
Gronwall's inequality then implies
\begin{eqnarray} \label{eq:bound-A:2}
\left|A\left(t; \frac{\theta}{n}\right)\right| \le \left|A\left(0; \frac{\theta}{n}\right)\right| \cdot e^{(\alpha - \beta)t} \cdot e^{\eta t}, \quad \text{for $t \in [0,c_n]$}.
\end{eqnarray}
Now notice that for $A\left(0; \frac{\theta}{n}\right) = \theta/n$ and $t = t_nT$ where $T<1$,
the right--hand--side of the above inequality becomes
\[ |\theta| \cdot n^{T(1+ \frac{\eta}{\alpha - \beta} ) -1} \]
which is smaller than $\delta$ when $n$ is large and $\eta$ is set sufficiently small. Hence when $n$ is large, we have $t_nT \le c_n$ and thus $\left|A\left(t; \frac{\theta}{n}\right)\right| \le \delta$ for all $t \in [0, t_n T]$.

Next we prove
\begin{eqnarray} \label{eq:B}
\lim_{n\rightarrow\infty}\frac{1}{n^{T}}\cdot B\left(t_{n}T;\frac{\theta}{n}\right) = 0.
\end{eqnarray}
Recall from the ODE for $B$ that
\begin{equation*}
B\left(t_{n}T;\frac{\theta}{n}\right)  = \mu \int_{0}^{t_{n}T} \left(e^{\alpha A\left(s;\frac{\theta}{n}\right)} -1 \right) ds.
\end{equation*}
Note that for $n$ large, we have $|A(t; \frac{\theta}{n})| \le \delta$ for all $t \in [0, t_n T]$.
Together with the fact that $|e^{\alpha x} - 1 -\alpha x| \le \eta |x|$ when $|x| \le \delta$, we deduce that
\begin{eqnarray*}
\left|B\left(t_{n}T;\frac{\theta}{n}\right)\right| &\le& \mu (\alpha + \eta) \int_{0}^{t_{n}T} \left| A\left(s;\frac{\theta}{n}\right) \right| ds\\
& \le & \mu (\alpha + \eta) \cdot \frac{|\theta|}{n} \cdot  \int_{0}^{t_{n}T}  e^{(\alpha - \beta+ \eta)t} dt\\
& =&  \mu (\alpha + \eta) \cdot \frac{|\theta|}{n} \cdot \left( n^{T(1+ \frac{\eta}{\alpha - \beta} )} -1  \right).
\end{eqnarray*}
where in the second inequality we have used \eqref{eq:bound-A:2} and $A(0; \frac{\theta}{n})= \frac{\theta}{n}$.
Thus \eqref{eq:B} readily follows. So the proof of part (i) is complete after applying
the G\"{a}rtner-Ellis Theorem.

(ii) We next prove part (ii). The proof is similar to that of part (i), so we only outline the key steps.
Recall that
\begin{equation*}
\mathbb{E}\left[e^{\frac{\theta}{n}N_{t_{n}T}}\right]=e^{-\frac{\theta}{\alpha n}n}
e^{C(t_{n}T;\frac{\theta}{\alpha n})n+D(t_{n}T;\frac{\theta}{\alpha n})},
\end{equation*}
where $C, D$ solve the ODEs in \eqref{eq:ODE-C} and \eqref{eq:ODE-D}
with initial condition $C(0; \frac{\theta}{\alpha n})=\frac{\theta}{\alpha n} $ and $D(0;\frac{\theta}{\alpha n})=0$.

It suffices to show
\begin{eqnarray}
\lim_{n\rightarrow\infty}{n^{1-T}} C\left(t_{n}T;\frac{\theta}{\alpha n}\right) &= & \frac{1}{\alpha-\beta}\theta, \label{eq:C-lim}\\
\lim_{n\rightarrow\infty}{n^{-T}} D\left(t_{n}T;\frac{\theta}{\alpha n}\right) &= & 0, \label{eq:D-lim}
\end{eqnarray}
When $C(0; \frac{\theta}{n \alpha})= \frac{\theta}{n \alpha} $, we look at the Riccati equation below to obtain estimates for $C$:
\begin{align}
&z'(t)= (\alpha-\beta) z(t)+ K z^2(t)+ \frac{\beta\theta}{n \alpha},\label{eq:recatti-z}
\\
&z(0)= \frac{\theta}{n \alpha}. \nonumber
\end{align}
When $n$ is large, we have $(\alpha-\beta)^2 > 4 K \cdot \frac{\beta\theta}{n \alpha} $. This implies the constant function
\begin{eqnarray*}
\bar z := \frac{1}{2K} \left( \sqrt{(\alpha-\beta)^2 - 4 K \cdot \frac{\beta\theta}{n \alpha}} - (\alpha -\beta) \right)
\end{eqnarray*}
is a particular solution to the Riccati equation \eqref{eq:recatti-z}. In addition, an application of Taylor expansion yields
\begin{eqnarray*}
\bar z = \frac{-\beta}{\alpha -\beta} \cdot \frac{\theta}{n \alpha} + O(n^{-2}), \quad \text{as $n \rightarrow \infty.$}
\end{eqnarray*}
Write $v(t)=z(t)-\bar z$. Then it is clear that
\begin{eqnarray} \label{eq:v0}
v(0)= z(0)- \bar z = \frac{1}{\alpha - \beta} \cdot \frac{\theta}{n} + O(n^{-2}) .
\end{eqnarray}
Moreover, one readily verifies that $v$ satisfies the Bernoulli equation
\[v'(t)=(\alpha- \beta + 2K \bar z) v(t) + K v^2(t),\]
which implies that
\begin{eqnarray}\label{eq:vt}
v(t) = \left( \left(\frac{1}{v(0)} + \frac{K}{\alpha - \beta + 2K \bar z}\right) \cdot e^{(\beta - \alpha - 2K \bar z) t} - \frac{K}{\alpha - \beta + 2K \bar z} \right)^{-1}.
\end{eqnarray}
Suppose that for $n$ large, we have $|C(t; \frac{\theta}{\alpha n})| \le \delta$ for all $t \in [0, t_n T]$. Then we obtain from Gronwall's inequality for nonlinear ODEs that
\[C(t_nT) \le z(t_n T) =v(t_n T) + \bar z.  \]
Together with \eqref{eq:v0} and \eqref{eq:vt}, we deduce that
\begin{eqnarray*}
\limsup_{n\rightarrow\infty}{n^{1-T}} C\left(t_{n}T;\frac{\theta}{\alpha n}\right) \le \limsup_{n\rightarrow\infty}{n^{1-T}} (v(t_n T) + \bar z)=  \frac{1}{\alpha-\beta}\theta.
\end{eqnarray*}
A similar argument leads to
\begin{eqnarray*}
\liminf_{n\rightarrow\infty}{n^{1-T}} C\left(t_{n}T;\frac{\theta}{\alpha n}\right) \ge \frac{1}{\alpha-\beta}\theta.
\end{eqnarray*}
Hence the proof of \eqref{eq:C-lim} is complete if we can verify that $|C(t; \frac{\theta}{\alpha n})| \le \delta$ for all $t \in [0, t_n T]$. Similarly as before, this can be done using the facts that $|g(x)| \le \eta x$ for $x$ small and applying Gronwall's inequality to obtain the following bound on the function $C$:
\begin{eqnarray*} \label{eq:bound-C}
\left|C\left(t;\frac{\theta}{n \alpha}\right)\right| \le \left|C\left(0;\frac{\theta}{n \alpha}\right)\right|
\cdot \left(1+ \beta \int_{0}^t e^{(\beta-\alpha)s }ds\right) \cdot e^{(\alpha - \beta)t} \cdot e^{\eta t}, \quad \text{for $t \in [0,d_n]$}.
\end{eqnarray*}
where $d_n = \sup\{t \ge0 : |C(t; \frac{\theta}{\alpha n})| \le \delta, \quad \text{for all $s \le t$} \}$.
In addition,
the proof of \eqref{eq:D-lim} follows similarly as for the proof of \eqref{eq:B}.
The proof is complete after applying the G\"{a}rtner-Ellis Theorem.
\end{proof}

\subsection{Proofs of results in Section~\ref{sec:subcritical}}\label{sec:proof-subcritical}
\begin{proof}[Proof of Theorem~\ref{ScalarN}]  
We apply G\"{a}rtner-Ellis theorem. The key idea is to study asymptotic behavior of the solutions of the ODEs \eqref{eq:ODE-C} and \eqref{eq:ODE-D} that characterize the moment generating function of $N_t$.

Recall from \eqref{eq:v} that for $Z_0=n$, we have
\begin{eqnarray}\label{eq:ntN}
\frac{1}{n}\log\mathbb{E}[e^{\theta N_{nT}}]
= C\left(nT; \frac{\theta}{\alpha}\right) - \frac{\theta}{\alpha} + \frac{1}{n} D\left(nT; \frac{\theta}{\alpha}\right),
\end{eqnarray}
with initial condition $C(0;\frac{\theta}{\alpha})=\frac{\theta}{\alpha}$ and $D(0;\frac{\theta}{\alpha})=0$.
Hence to study the limit of \eqref{eq:ntN} as $n\rightarrow\infty$ and then apply G\"{a}rtner-Ellis theorem, we need to look at the asymptotic behavior of the
ODE solutions $C$ and $D$ as $t \rightarrow \infty$.

To this end, let
\begin{equation*}
F(x):=-\beta x+e^{\alpha x}-1+\frac{\theta\beta}{\alpha}.
\end{equation*}
Then Equation~\eqref{eq:ODE-C} becomes $C'(t;\frac{\theta}{\alpha})=F(C(t;\frac{\theta}{\alpha}))$. It is clear that
$F$ is a convex function and $F(\pm\infty)=\infty$. In addition,
$F(x)$ achieves its minimum at $x=\frac{1}{\alpha}\log\frac{\beta}{\alpha}$ at which $F'(x)=0$
and
\begin{equation*}
F\left(\frac{1}{\alpha}\log\frac{\beta}{\alpha}\right)
=-\frac{\beta}{\alpha}\log\frac{\beta}{\alpha}+\frac{\beta}{\alpha}-1+\frac{\theta\beta}{\alpha}.
\end{equation*}
Thus, $\min_{x}F(x)\leq 0$ if and only if $\theta\leq\theta_{c}:=\frac{\alpha}{\beta}-\log\frac{\alpha}{\beta}-1$.
When $\theta\leq\theta_{c}$, since $\alpha< \beta$, one readily verifies that $F'(C(0;\frac{\theta}{\alpha})) <0$. Hence the ODE solution
$C(t;\frac{\theta}{\alpha})$ converges to the smaller solution $x^{\ast}(\theta)$ of the equation
\begin{equation*}
F(x)=-\beta x+e^{\alpha x}-1+\frac{\theta\beta}{\alpha}=0,
\end{equation*}
as $t\rightarrow\infty$. Note from \eqref{eq:ODE-C} and \eqref{eq:ODE-D} we have
\begin{equation}\label{eq:D-C}
D\left(t;\frac{\theta}{\alpha}\right)=\mu\cdot \left(C\left(t;\frac{\theta}{\alpha}\right)-C\left(0;\frac{\theta}{\alpha}\right)\right)
+\mu\beta\int_{0}^{t}C\left(s;\frac{\theta}{\alpha}\right)ds-\frac{\mu\theta\beta}{\alpha}t.
\end{equation}
Therefore, $\frac{D(t;\frac{\theta}{\alpha})}{t}\rightarrow\mu\beta x^{\ast}(\theta)-\frac{\mu\theta\beta}{\alpha}$
as $t\rightarrow\infty$ and for any $\theta\leq\theta_{c}$,
\begin{eqnarray*}
\lim_{n\rightarrow\infty}\frac{1}{n}\log\mathbb{E}[e^{\theta N_{nT}}]
&=& \lim_{n\rightarrow\infty} \left[ C\left(nT; \frac{\theta}{\alpha}\right) - \frac{\theta}{\alpha} + \frac{1}{n} D\left(nT; \frac{\theta}{\alpha}\right)   \right] \\
&=& -\frac{\theta}{\alpha}+x^{\ast}(\theta)+\left(\mu\beta x^{\ast}(\theta)-\frac{\mu\theta\beta}{\alpha}\right)T,
\end{eqnarray*}
and $\lim_{n\rightarrow\infty}\frac{1}{n}\log\mathbb{E}[e^{\theta N_{nT}}]=+\infty$ otherwise.
For $\theta=\theta_{c}$, we get $-\beta x^{\ast}(\theta_{c})+e^{\alpha x^{\ast}(\theta_{c})}-\frac{\beta}{\alpha}\log\frac{\alpha}{\beta}
-\frac{\beta}{\alpha}=0$, which implies that $x^{\ast}(\theta_{c})=\frac{\log(\beta/\alpha)}{\alpha}$.
By differentiating the equation $-\beta x^{\ast}(\theta)+e^{\alpha x^{\ast}(\theta)}-1+\frac{\theta\beta}{\alpha}=0$
with respect to $\theta$, we get
\begin{equation*}
\frac{d}{d\theta}x^{\ast}(\theta)=\frac{(\beta/\alpha)}{\beta-\alpha e^{\alpha x^{\ast}(\theta)}}\rightarrow\infty,
\qquad\text{as $\theta\rightarrow\theta_{c}$}.
\end{equation*}
Thus, we verified the essential smoothness condition.
By G\"{a}rtner-Ellis theorem,
$\mathbb{P}(\frac{N_{nT}}{n}\in\cdot)$ satisfies
a large deviation principle with the rate function
\begin{equation}\label{PlugVI}
I(x)=\sup_{\theta\in\mathbb{R}}\left\{\theta x
+\frac{\theta}{\alpha}-x^{\ast}(\theta)-\left(\mu\beta x^{\ast}(\theta)-\frac{\mu\theta\beta}{\alpha}\right)T\right\}.
\end{equation}

We next solve the optimization problem in \eqref{PlugVI} and simplify the rate function above.
At the optimal $\theta$ in \eqref{PlugVI}, we have
\begin{equation*}
x+\frac{1}{\alpha}-\frac{d}{d\theta}x^{\ast}(\theta)
-\mu\beta\frac{d}{d\theta}x^{\ast}(\theta)T+\frac{\mu\beta}{\alpha}T=0,
\end{equation*}
which implies that
\begin{equation}\label{PlugI}
\frac{d}{d\theta}x^{\ast}(\theta)
=\frac{x+\frac{1}{\alpha}+\frac{\mu\beta}{\alpha}T}{1+\mu\beta T}.
\end{equation}
Recall that $x^{\ast}(\theta)$ satisfies
\begin{equation}\label{PlugIII}
-\beta x^{\ast}(\theta)+e^{\alpha \cdot x^{\ast}(\theta)}-1+\frac{\theta\beta}{\alpha}=0.
\end{equation}
Differentiating with respect to $\theta$, we get
\begin{equation}\label{PlugII}
-\beta\frac{d}{d\theta}x^{\ast}(\theta)+\alpha\frac{d}{d\theta}x^{\ast}(\theta)e^{\alpha x^{\ast}}
+\frac{\beta}{\alpha}=0.
\end{equation}
Plugging \eqref{PlugI} into \eqref{PlugII}, we get
\begin{equation}\label{PlugV}
x^{\ast}(\theta)=\frac{1}{\alpha}\log\left(\frac{\beta x}{\alpha x+1+\mu\beta T}\right).
\end{equation}
Plugging this into \eqref{PlugIII}, we get the optimal $\theta$ in \eqref{PlugVI} is given by
\begin{equation}\label{PlugIV}
\theta=\log\left(\frac{\beta x}{\alpha x+1+\mu\beta T}\right)-\frac{\alpha x}{\alpha x+1+\mu\beta T}+\frac{\alpha}{\beta}.
\end{equation}
Finally, substituting \eqref{PlugV} and \eqref{PlugIV} into \eqref{PlugVI}, we get
\begin{equation*}
I(x)=x\log\left(\frac{\beta x}{\alpha x+1+\mu\beta T}\right)
-x+\frac{\alpha x+1+\mu\beta T}{\beta}.
\end{equation*}
The proof is therefore complete.
\end{proof}

\begin{proof}[Proof of Theorem~\ref{ScalarN2}] The proof is similar to that of Theorem \ref{ScalarN}, so we only provide a sketch.
From the proof of Theorem \ref{ScalarN}, we have
\begin{equation*}
\mathbb{E}\left[e^{\theta N_{t_{n}T}}\right]
=e^{-\frac{\theta}{\alpha}n}e^{C(t_{n}T; \frac{\theta}{\alpha})n+D(t_{n}T;\frac{\theta}{\alpha})},
\end{equation*}
where for $\theta \le \theta_c:=\frac{\alpha}{\beta}-\log\frac{\alpha}{\beta}-1$, $\lim_{n\rightarrow\infty}C(t_{n}T;\frac{\theta}{\alpha})=x^{\ast}(\theta)$, the smaller solution to the equation
$-\beta x+e^{\alpha x}-1+\frac{\theta\beta}{\alpha}=0$. In addition, we infer from \eqref{eq:D-C} that
\begin{equation*}
D\left(t_{n}T; \frac{\theta}{\alpha} \right)=\mu\left(C\left(t_{n}T;\frac{\theta}{\alpha} \right)-\frac{\theta}{\alpha}\right)
+\mu\beta\int_{0}^{t_{n}T}C\left(s;\frac{\theta}{\alpha} \right)ds-\frac{\mu\theta\beta}{\alpha}t_{n}T.
\end{equation*}
Therefore, if $\lim_{n\rightarrow\infty}\frac{t_{n}}{n}=0$, we have for $\theta \le \theta_c$,
\begin{eqnarray*}
\lim_{n\rightarrow\infty}\frac{1}{n}\log\mathbb{E}\left[e^{\theta N_{t_{n}T}}\right]
&=& \lim_{n\rightarrow\infty} \left[ C\left(t_n T; \frac{\theta}{\alpha}\right) - \frac{\theta}{\alpha} + \frac{1}{n} D\left(t_nT; \frac{\theta}{\alpha}\right)   \right] \\
&=& -\frac{\theta}{\alpha}+x^{\ast}(\theta),
\end{eqnarray*}
and $\lim_{n\rightarrow\infty}\frac{1}{n}\log\mathbb{E}\left[e^{\theta N_{t_{n}T}}\right] = \infty$ otherwise.
Similarly, if $\lim_{n\rightarrow\infty}\frac{t_{n}}{n}=\infty$, we have for $\theta \le \theta_c$,
\begin{eqnarray*}
\lim_{n\rightarrow\infty}\frac{1}{t_{n}}\log\mathbb{E}\left[e^{\theta N_{t_{n}T}}\right]
&=& \lim_{n\rightarrow\infty} \left[ \frac{n}{t_n} \cdot \left(C\left(t_n T; \frac{\theta}{\alpha}\right) - \frac{\theta}{\alpha} \right) + \frac{1}{t_n} D\left(t_nT; \frac{\theta}{\alpha}\right)   \right] \\
&=& \mu\beta x^{\ast}(\theta)T-\frac{\mu\theta\beta}{\alpha}T,
\end{eqnarray*}
and $\lim_{n\rightarrow\infty}\frac{1}{t_{n}}\log\mathbb{E}\left[e^{\theta N_{t_{n}T}}\right] = \infty$ otherwise.
We can also check that $\frac{d}{d\theta}x^{\ast}(\theta)\rightarrow\infty$ as $\theta\rightarrow\theta_{c}$.
Therefore, by G\"{a}rtner-Ellis theorem and following the proof of Theorem \ref{ScalarN}, we have
proved the desired results.
\end{proof}


\begin{proof}[Proof of Theorem \ref{subZLDP}]
The proof is similar as the proof of Theorem \ref{thm:supercritical}, so we only provide a sketch.
From \eqref{eq:u} we have
\begin{equation*}
\mathbb{E}\left[e^{\frac{\theta}{n^{\gamma}}Z_{t_{n}T}}|Z_{0}=n\right] = e^{A(t_{n}T; \frac{\theta}{n^{\gamma}}) n + B(t_{n}T; \frac{\theta}{n^{\gamma}})}.
\end{equation*}
Similar as in the proof of Theorem \ref{thm:supercritical}, we can show that
\begin{equation*}
\lim_{n\rightarrow\infty}\frac{1}{n^{1-\gamma-T}}\log\mathbb{E}[e^{\frac{\theta}{n^{\gamma}}Z_{t_{n}T}}|Z_{0}=n]=\theta.
\end{equation*}
The result then follows from G\"{a}rtner-Ellis theorem.
\end{proof}

\end{document}